\documentclass[12pt,reqno]{amsart}
\usepackage{amsmath,amsfonts,amsthm,graphicx}
\usepackage{a4}
\usepackage[mathscr]{euscript}

\title{Robust error bounds for finite element approximation of reaction-diffusion problems with non-constant reaction coefficient in arbitrary space dimension}

\author{Mark Ainsworth and Tom\'a\v s Vejchodsk\'y}

\address{
  Mark Ainsworth,
  Division of Applied Mathematics,
  Brown University,
  182 George Street
  Providence, RI 02912, USA
}\email{mark\_ainsworth@brown.edu}
\address{
  Tom{\'a}{\v s} Vejchodsk{\'y},
  Mathematical Institute,
  University of Oxford,
  Andrew Wiles Building,
  Radcliffe Observatory Quarter,
  Woodstock Road,
  Oxford, OX2 6GG, UK,
  and
  Institute of Mathematics,
  Academy of Sciences,
  {\v Z}itn{\'a} 25, CZ-115 67 Prague 1,
  Czech Republic.
}\email{vejchod@math.cas.cz}
    \thanks{Date: \today.\\
    M.A. acknowledges the support from AFOSR contract FA 9550-12-1-0399.
  T.V. acknowledges the support from RVO~67985840.
  Further, his research was supported by Marie Curie Intra European Fellowship
  within the 7th European Community Framework Programme, project no.~328008.
    }

\begin{document}
\maketitle
\markright{ROBUST ERROR BOUNDS IN ARBITRARY SPACE DIMENSION}

\begin{abstract}
We present a fully computable a posteriori error estimator for piecewise linear
finite element approximations
of reaction-diffusion problems with mixed boundary conditions and piecewise
constant reaction coefficient formulated in arbitrary dimension.
The estimator provides a guaranteed upper bound on the energy norm of the error
and it is robust for all values of the reaction coefficient,
including the singularly perturbed case.
The approach is based on robustly equilibrated boundary flux functions
% Lind et al. (J. Comp. Phys. 231 (2012) 1499--1523)
%\cite{AinOde:2000}
of Ainsworth and Oden (Wiley 2000)
and on subsequent robust and explicit flux reconstruction.
This paper simplifies and extends the applicability of the previous result
%\cite{robustaee:2010}
of Ainsworth and Vejchodsk{\'y} (Numer. Math. 119 (2011) 219--243)
in three aspects: (i) arbitrary dimension, (ii) mixed boundary conditions,
and (iii) non-constant reaction coefficient.
It is the first robust upper bound on the error with these properties.
An auxiliary result that is of independent interest is the derivation of new
explicit constants for two types of trace inequalities on simplices.
\end{abstract}

\renewcommand{\b}[1]{\boldsymbol{#1}}
\newcommand{\bc}{\b{c}}
\newcommand{\bfx}{\b{x}}
\newcommand{\bn}{\b{n}}
\newcommand{\bnabla}{\b{\nabla}}
\newcommand{\bphi}{\b{\varphi}}
\newcommand{\bt}{\b{t}}
\newcommand{\btau}{\b{\tau}}
\newcommand{\btauKL}{\btau_K^{\mathrm{L}}}
\newcommand{\btauKO}{\btau_K^{\mathrm{O}}}
\newcommand{\btauKQ}{\btau_K^{\mathrm{Q}}}
\newcommand{\bx}{\mbox{\boldmath $x$}}
\newcommand{\cB}{\mathcal{B}}
\newcommand{\cC}{\mathscr{C}} %{\mathcal{C}}
\newcommand{\cF}{\mathcal{F}}
\newcommand{\cG}{\mathcal{G}}
\newcommand{\cN}{\mathcal{N}}
\newcommand{\cT}{\mathcal{T}}
\newcommand{\CT}{C_\mathrm{T}}
\newcommand{\CTkappa}{C_\mathrm{T}^\kappa}
\newcommand{\CTloc}{C_\mathrm{T}^\mathrm{loc}}
\newcommand{\ddiv}{\operatorname{div}}
\newcommand{\ds}{\dx[\b{s}]}
\newcommand{\dx}[1][\bfx]{\,\mathrm{d}#1}
\newcommand{\GammaD}{{\Gamma_{\mathrm{D}}}}
\newcommand{\GammaN}{{\Gamma_{\mathrm{N}}}}
\newcommand{\gD}{g_{\mathrm{D}}}
\newcommand{\gN}{g_{\mathrm{N}}}
\newcommand{\Hdiv}[1][\Omega]{\b{H}(\ddiv,#1)}
\newcommand{\kappamax}[1][\widetilde K]{\kappa_{#1,\mathrm{max}}}
\newcommand{\kappamin}[1][\widetilde K]{\kappa_{#1,\mathrm{min}}}
\newcommand{\norm}[1]{\left\|#1\right\|}
\newcommand{\oCT}{\overline{C}_\mathrm{T}}
\newcommand{\oGammaD}{\overline\Gamma_{\mathrm{D}}}
\newcommand{\oGammaN}{\overline\Gamma_{\mathrm{N}}}
\newcommand{\osc}{\operatorname{osc}}
\newcommand{\R}{\mathbb{R}}
\newcommand{\rhomax}{\tilde\rho_K}
\newcommand{\trinorm}[1]{|\!|\!|#1|\!|\!|}
\newcommand{\ttg}[1]{\textrm{(\ref{#1})}}

\newtheorem{theorem}{Theorem}%[section]
\newtheorem{lemma}[theorem]{Lemma}%[section]

\section{Introduction}

Consider a linear reaction-diffusion problem in a domain
$\Omega \subset \R^d$ with
mixed boundary conditions:
\begin{equation}
\label{eq:modpro}
  -\Delta u + \kappa^2 u = f \quad\text{in }\Omega; \qquad
  u = 0 \quad\text{on }\GammaD; \qquad
  \partial u/\partial \bn = \gN \quad\text{on }\GammaN,
\end{equation}
where $\bn$ stands for the unit outward normal vector to the boundary
$\partial\Omega$.
The dimension $d \geq 2$ is chosen arbitrarily.
For simplicity we assume $\Omega$ to be a polytope.
The portions $\GammaD$ and $\GammaN$ of the boundary $\partial\Omega$
are open, disjoint, and satisfy $\oGammaD \cup \oGammaN = \partial\Omega$.
The reaction coefficient $\kappa \geq 0$  is considered to be piecewise constant.
In order to guarantee unique solvability of \ttg{eq:modpro}, we consider
$\kappa > 0$ in a subdomain of $\Omega$ of a positive measure
or a positive measure of $\GammaD$.
We use the finite element method  to approximate the exact solution $u$ by a piecewise affine function $u_h$ with respect to a simplicial partition $\cT_h$
of $\Omega$.

In this paper we derive a computable a posteriori error estimate based
on robust flux equilibration and explicit flux reconstruction.
This error estimate $\eta$
provides a guaranteed and fully computable
upper bound on the energy norm of the error $\trinorm{u-u_h}$
and it is robust with respect to both $\kappa$ and the mesh-size $h$.

% The a posteriori error estimation for elliptic boundary value problems
% has already long history. Modern research in the context of the finite
% element method started by pioneering work of Babu\v{s}ka and Rheinboldt
% \cite{BabRhe:1978b}. Since then many techniques and concepts were developed.
% Let us mention
% explicit and implicit residual methods,
% hierarchical estimates,
% estimates based on postprocessing,
% goal oriented approaches,
% and complementarity techniques.
% See books \cite{AinOde:2000,BabStr:2001,BanRan:2003,Rep:2008,Verfurth:1996}
% for more information.
% Naturally, these approaches yield error estimates with different properties.

A posteriori error estimates are useful for adaptive algorithms,
where they play two roles. Firstly, they
indicate where the computational mesh should be refined or coarsened.
Secondly, they provide quantitative information about
the size of the error for reliable stopping criterion.
Unfortunately, many existing estimators do not provide actual numerical bounds
that can be used as a stopping criterion.

Adaptive algorithms are convergent \cite{Dorfler:1996}
provided the error estimates are
\emph{locally efficient} and \emph{reliable}.
If $\eta_K$ stand for local error indicators on elements $K \in \cT_h$ and
$\eta^2 = \sum_{K\in\cT_h} \eta_K^2$ is the global error estimator,
then
the indicators $\eta_K$ are said to be \emph{locally efficient} if
there exists a constant $c>0$ such that
$$
  c\eta_K \leq \trinorm{u-u_h}_{\widetilde K},
$$
where $\trinorm{u-u_h}_{\widetilde K}$ stands for the energy norm restricted
to a patch $\widetilde K$ of elements consisting of $K$ and neighbouring
elements sharing at least one vertex with $K$.
Similarly, the error estimator $\eta$ is \emph{reliable} if there exists
a constant $C>0$ such that
$$
  \trinorm{u-u_h} \leq C \eta.
$$
The error estimate $\eta$ is \emph{robust} if the constants $c$ and $C$ are
independent of $\kappa$ and mesh-size $h$.
The error estimate $\eta$ is a \emph{guaranteed upper bound} if
$\trinorm{u-u_h} \leq \eta$, i.e. the reliability constant $C$ is equal to one.
Finally, the error bound $\eta$ is \emph{fully computable} if it can be evaluated
in terms of the approximation $u_h$ and given data
without the need for generic (unknown) constants.

A robust, reliable, locally efficient explicit a posteriori error
estimate for problem \ttg{eq:modpro} was first derived
by Verf\"urth in \cite{Verfurth:1998a,Verfurth:1998}. This estimate, however,
does not provide guaranteed upper bound on the error.
An estimator which does provide an upper bound along with robust local efficiency
was obtained by Ainsworth and Babu\v{s}ka in
\cite{AinBab:1999}, but this upper bound depends on an exact solution of
a Neumann problem and as noted in \cite{AinBab:1999} is not fully computable.
Subsequently in \cite{robustaee:2010} we were able to develop fully computable
error bounds in the two dimensional setting 
by a complementarity technique combined with robustly equilibrated fluxes
and explicit flux reconstruction.
Here, we develop a simpler flux reconstruction
that is suitable for any dimension $d\geq 2$
and is applicable to the case of piecewise constant
coefficient $\kappa$ including the situation where $\kappa$ can be very large
in some parts of the domain and vanishingly small in others.
Furthermore, we extend the previous results by considering nonhomogeneous Neumann boundary conditions.
In order to achieve these goals, we develop some 
new techniques and tools for the analysis
that are of wider applicability than the problem addressed here.

The question of robust a posteriori error estimates for singularly perturbed
problems is studied by other authors as well.
In \cite{Grosman:2006}, an error estimate that is robust with respect to anisotropic
meshes is obtained, but unfortunately does not provide guaranteed upper
bound on the error. A robust, locally efficient and fully computable guaranteed
upper bound was obtained in \cite{CheFucPriVoh:2009}
for the finite volume method and $d=2$ and $3$.
Recently, a robust estimator for the error in the maximum norm was obtained
in \cite{Linss:2012} for the case $d=1$.

The basic idea behind our work can be traced back
to the method of the hypercircle \cite{Synge:1957} and later to
\cite{AubBur:1971,HasHla:1976,Veubeke:1965}.
This approach has been adopted by Repin \cite{Rep:2008} and his group
for a wide class of problems in conjunction with the solution of a global
minimization problem to compute the error bound.
We avoid any global computations and instead develop 
local algorithms for guaranteed and fully computable error bounds
based on \emph{flux equilibration}
\cite{AinBab:1999,BraSch:2008,CaiZha:2010,JirStrVoh:2010,Kelly:1984,LadLeg:1983,ParesSantosDiez:2009,Vohralik:2011}
etc.
In the present work we will utilize the robust flux equilibration
from \cite{AinBab:1999}.

The rest of the paper is organized as follows.
Section~\ref{se:modpro} defines the finite element approximation and corresponding
assumptions. The core of the paper lies in Section~\ref{se:AEE},
where we present new trace inequalities on simplices,
and develop two new flux reconstructions both of which are used to
derive the a posteriori error and our main result.
Finally, Section~\ref{se:numex} provides an
illustrative numerical example and Section~\ref{se:conclusions}
draws the conclusions.

\section{Model Problem and Its Approximate Solution}
\label{se:modpro}

The weak formulation of \ttg{eq:modpro} reads:
find $u\in V = \{ v \in H^1(\Omega) : v = 0 \text{ on }
\GammaD \}$ such that
\begin{equation}
\label{eq:weakf}
  \cB(u, v) = \cF(v) \quad \forall v \in V,
\end{equation}
where $\cB$ and $\cF$ are bilinear and linear forms, respectively,
defined on $V$ by
\begin{equation*}
%\label{eq:defBF}
  \cB(u,v) = \int_\Omega (\bnabla u \cdot \bnabla v + \kappa^2 u v ) \dx; \quad
  \cF(v) = \int_\Omega f v \dx + \int_\GammaN \gN v \ds.
\end{equation*}

In order to discretize problem \ttg{eq:modpro}
we consider a family of partitions $\cG = \{\cT_h\}$
of the domain $\Omega$. Each partition $\cT_h$ consists of simplices (elements),
their union is $\overline\Omega$, their interiors are
pairwise disjoint, and every facet of each simplex lies either
in $\partial\Omega$ or it is completely shared by exactly two neighbouring
simplices.
We assume that all partitions in $\cG$ are compatible the coefficient $\kappa$
meaning that $\kappa$ is a constant $\kappa_K$ in any
element $K$ of $\cT_h$ for all $\cT_h \in \cG$.
%The constant value of $\kappa$ in an element $K$ is denoted by $\kappa_K$.

We denote by $h_K$, $\bx_K$, $\rho_K$, and $\bn_K$ the diameter, the \emph{incentre},
the \emph{inradius} of simplex $K$,
and the unit outward-facing normal vector to the boundary $\partial K$,
respectively.
The family of partitions $\cG$ is assumed to be regular, i.e.
there exists a constant $C>0$ such that
\begin{equation}
\label{eq:shapereg}
  \sup\limits_{\cT_h \in \cG} \max\limits_{K \in \cT_h} \frac{h_K}{\rho_K} \leq C,
\end{equation}
but is not requested to be quasi-uniform, thereby permitting the use of locally
refined meshes.
Throughout the paper we use symbol $C$ for a generic constant
whose value is independent of $\kappa$ and any mesh-size and whose
actual numerical value can differ in different occurrences.
Furthermore, we define
\begin{equation}
\label{eq:patchK}
  \widetilde K = \operatorname{int} \left\{ \bigcup \overline{K'} :
  \overline{K'} \cap \overline K \neq \emptyset \right\}
\end{equation}
to be the patch consisting of $K$ and elements sharing at least one
common point with $K$.

The regularity assumption implies several facts that we will use
in the subsequent analysis.
Firstly, the number of elements in any patch is
uniformly bounded over the family $\cG$ as is the number of patches
containing a particular element.
Secondly, within each patch $\widetilde K$, a local quasi-uniformity condition
$ch_K \leq h_{K'} \leq C h_K$ holds for all elements $K' \subset \widetilde K$
with uniform constants $c>0$ and $C>0$ over the family $\cG$.
Thirdly, the elements are shape regular meaning that there exists a positive
constant $\cC_0$ such that
\begin{equation}
\label{eq:shaperegrhoK}
  \frac{1}{\cC_0} \rho_K \leq \rho_{K'} \leq \cC_0 \rho_K
\end{equation}
for all elements $K' \subset \widetilde K$, all $K\in\cT_h$, and all $\cT_h \in \cG$.

The coefficient $\kappa$ is assumed to be piecewise constant, and such that for
some constant $C>0$ the following conditions
hold for all triangulations $\cT_h\in \cG$ and all elements $K\in\cT_h$:
\begin{align}
  \label{eq:kappacond1}
  &\text{if}\quad \kappa_K\neq 0 \quad\text{then}\quad
  \kappa_{K'} \leq C \kappa_K \quad \text{for all }K'\subset\widetilde K;
\\
  \label{eq:kappacond2}
  &\text{if}\quad \kappa_K= 0 \quad\text{then}\quad
  \kappa_{K'} \leq C \quad \text{for all }K'\subset \widetilde K.
\end{align}
These assumptions rule out the case of arbitrarily high jumps
in values $\kappa_K$ between neighbouring elements.
% On the other hand, it is easy to see that these assumptions are satisfied for example if
% $\kappa_K = |K|^{-1} \int_K \kappa(\bfx) \dx$ for a nonnegative function
% $\kappa \in C^{\infty}(\overline\Omega)$.

One consequence of assumptions \ttg{eq:kappacond1}--\ttg{eq:kappacond2}
together with the quasi-uniformity of $h_K$ is
existence of a constant $C>0$ such that
for all $\cT_h\in\cG$, all $K \in \cT_h$, and all elements
$K' \subset \widetilde K$, we have
\begin{equation}
\label{eq:mincond}
 C^{-1} \min\{h_{K'},\kappa_{K'}^{-1} \} \leq
        \min\{h_{K},\kappa_{K}^{-1} \} \leq
      C \min\{h_{K'},\kappa_{K'}^{-1} \}.
\end{equation}
The quantity $\min\{h_{K},\kappa_{K}^{-1} \}$ appears extensively
throughout the paper, and for the avoidance of doubt, we note explicitly that
\begin{equation}
\label{eq:kappa0}
   \min\{h_{K},\kappa_{K}^{-1} \} = h_K
   \quad\text{if }\kappa_K = 0.
\end{equation}

Let $X_h$ be the space of continuous and piecewise affine functions with respect
to the partition $\cT_h$, and consider the subspace
$V_h = \{ v_h \in X_h : v_h = 0 \text{ on } \GammaD \}$.
The finite element approximation $u_h \in V_h$ of \ttg{eq:modpro} is then given by
\begin{equation}
  \label{eq:FEM}
  \cB(u_h, v_h ) = \cF(v_h) \quad \forall v_h \in V_h.
\end{equation}

Finally, we use local counterparts of the bilinear and linear forms defined by
$$
  \cB_K(u,v) = \int_K (\bnabla u \cdot \bnabla v + \kappa_K^2 u v ) \dx; \quad
  \cF_K(v) = \int_K f v \dx + \int_{\GammaN\cap\partial K} \gN v \ds.
$$
The associated global and local energy norms $\trinorm{\cdot}$ and $\trinorm{\cdot}_K$
are defined by $\trinorm{v}^2 = \cB(v,v)$ and $\trinorm{v}_K^2 = \cB_K(v,v)$,
respectively.
Analogously, we use $\norm{\cdot}$ and $\norm{\cdot}_K$ for
the $L^2(\Omega)$ and $L^2(K)$ norms, respectively.

\section{A Posteriori Error Estimator}\label{se:AEE}

\subsection{Trace inequalities on simplices}
The derivation of complementarity based error estimates
for problems with nonhomogeneous Neumann boundary conditions
requires certain types of trace inequalities.
Moreover, the constants appearing in these inequalities are present
in the final error bounds.
We derive two new trace inequalities
for simplices together with explicit formulas for the corresponding
constants.

\begin{lemma}\label{le:trace}
Let $K$ be a $d$-dimensional non-degenerate simplex and let $\gamma$ be one of its facets.
Let $h_K$ be the diameter of $K$ and $\kappa_K \geq 0$ a constant.
Let $v \in H^1(K)$ and let $\bar v_\gamma$ denote the average value of $v$ on $\gamma$.
Then
\begin{align}
\label{eq:trace2}
\norm{v}_{\gamma} &\leq \CT \trinorm{v}_K  \quad\text{for } \kappa_K > 0,
\\
\label{eq:trace1}
\norm{v - \bar v_\gamma}_{\gamma} &\leq \oCT \trinorm{v}_K,
\end{align}
hold with constants $\CT,\oCT > 0$ given by
\begin{align*}
  \CT^2 &= \frac{|\gamma|}{d|K|} \frac{1}{\kappa_K}
     \sqrt{(2h_K)^2 + (d/\kappa_K)^2},
\\
     \oCT^2 &= \frac{|\gamma|}{d|K|} \min\{h_K/\pi,\kappa_K^{-1}\}
  \left(2h_K + d \min\{h_K/\pi,\kappa_K^{-1}\} \right).
\end{align*}
\end{lemma}
\begin{proof}
Let $\bx_0$ be the vertex of $K$ opposite to the facet $\gamma$.
Define $\bphi(\bx) = \bx - \bx_0$ for $\bx \in K$.
Note that $\bn_K\cdot\bphi = 0$ on $\partial K \setminus \gamma$ and
$\bn\cdot\bphi = \tilde\varrho_K$ on $\gamma$,
where $\bn_K$ is the unit outward normal to $\partial K$ and
$\tilde\varrho_K$ is the distance between $\gamma$ and $\bx_0$, i.e. the altitude of $K$.
In particular, $\tilde\varrho_K = d |K|/|\gamma|$.

Let $v\in H^1(K)$ then
\begin{multline}
\label{eq:pt1}
  \frac{d |K|}{|\gamma|}  \norm{v}_\gamma^2
 =\int_{\gamma} v^2 \bn_K\cdot\bphi \ds
 =\int_{\partial K} v^2 \bn_K\cdot\bphi \ds
 = \int_K \ddiv(v^2\bphi) \dx
\\
 = 2\int_K v \bphi\cdot\bnabla v \dx + \int_K v^2 \ddiv \bphi \dx
 \leq \norm{v}_K \left( 2h_K \norm{\bnabla v}_K + d \norm{v}_K \right).
\end{multline}

% \begin{equation}
% %\label{eq:pt1}
%   \int_{\partial K} v^2 \bn_K\cdot\bphi \ds
%  =\int_{\gamma} v^2 \bn_K\cdot\bphi \ds
%  =\frac{d+1}{|\gamma|} |K| \norm{v}_\gamma^2
% \end{equation}
% and also
% \begin{multline}
% %\label{eq:pt2}
%   \int_{\partial K} v^2 \bn_K\cdot\bphi \ds
%  = \int_K \ddiv(v^2\bphi) \dx
%  = 2\int_K v \bphi\cdot\bnabla v \dx + \int_K v^2 \ddiv \bphi \dx
% \\
%  \leq \norm{v}_K \left( 2h_K \norm{\bnabla v}_K + d \norm{v}_K \right).
% \end{multline}

Using $\norm{v}_K \leq \kappa^{-1} \trinorm{v}_K$ and
$2h_K \norm{\bnabla v}_K + d \norm{v}_K \leq \left((2h_K)^2 + (d/\kappa_K)^2 \right)^{1/2}
\trinorm{v}_K$
in \ttg{eq:pt1}, we obtain \ttg{eq:trace2}.

Now, consider $\bar v_\gamma = |\gamma|^{-1} \int_{\gamma} v \ds$
and $\bar v_K = |K|^{-1} \int_K v \dx$.
Applying estimate \ttg{eq:pt1} to $v-\bar v_K$ yields
\begin{equation}
  \label{eq:pt2}
  \norm{v - \bar v_\gamma}_\gamma^2 \leq \norm{v - \bar v_K}_\gamma^2
  \leq \frac{|\gamma|}{d |K|} \norm{v - \bar v_K}_K
      \left( 2h_K \norm{\bnabla v}_K + d \norm{v - \bar v_K}_K \right).
\end{equation}
The norm $\norm{v - \bar v_K}_K$ can be bounded in either of the
two ways:
%\begin{align*}
$$
  \norm{v - \bar v_K}_K
  \leq \norm{v}_K \leq \kappa_K^{-1}\trinorm{v}_K
\quad\text{and}\quad
  \norm{v - \bar v_K}_K
  \leq \frac{h_K}{\pi} \norm{\bnabla v}_K \leq \frac{h_K}{\pi} \trinorm{v}_K,
$$
%\end{align*}
where we use Poincar\'e inequality \cite{PayWei:1960}.
Thus, $\norm{v-\bar v_K}_K \leq \min\{h_K/\pi,\kappa_K^{-1}\} \trinorm{v}_K$.
Using this estimate and inequality $\norm{\bnabla v}_K \leq \trinorm{v}_K$
in \ttg{eq:pt2}, we derive \ttg{eq:trace1}.
% 
% =====
% Poincar\'e inequality $\norm{v-\bar v_K}_K \leq (h_K/\pi) \norm{\bnabla v}_K$,
% see \cite{PayWei:1960}, and two-dimensional Cauchy-Schwarz inequality
% enable the following two bounds:
% \begin{align*}
%   2h_K \norm{\bnabla v}_K + d \norm{v - \bar v_K}_K
%   &\leq (2h_K + d h_K/\pi) \norm{\bnabla v}_K,
% \\
%   2h_K \norm{\bnabla v}_K + d \norm{v - \bar v_K}_K
%   &\leq \left( (2h_K)^2+(d/\kappa_K)^2 \right)^{1/2} \trinorm{v}_K
%    \leq (2h_K + d/\kappa_K) \trinorm{v}_K.
% \end{align*}
% Thus,
% \begin{equation}
%   \label{eq:pt3}
%   2h_K \norm{\bnabla v}_K + d \norm{v - \bar v_K}_K
%   \leq \left( 2h_K + d \min\{h_K/\pi,\kappa_K^{-1}\} \right) \trinorm{v}_K.
% \end{equation}
% Similarly, we derive estimate
% $\norm{v-\bar v_K}_K \leq \min\{h_K/\pi,\kappa^{-1}\} \trinorm{v}_K$,
% see also \cite[p.~228]{robustaee:2010}.
% Application of this estimate together with \ttg{eq:pt3}
% in \ttg{eq:pt2} yields \ttg{eq:trace1}.
% =====
\end{proof}

The constants $\CT$ and $\oCT$ from Lemma~\ref{le:trace}
have the correct asymptotic behaviour with respect to $h_K$ and $\kappa_K$,
but they are not optimal in terms of absolute values. Optimal values for trace constants
are not known in general, but their two-sided bounds can be computed numerically
for quite general domains \cite{SebVej:2013}.
Further, let us note that Lemma~\ref{le:trace} and its proof is similar to
the multiplicative trace inequality \cite{DolFeiSch:2002}.

\subsection{General framework}

We define
$\Pi_K : L^2(K) \rightarrow \mathbb{P}^1(K)$ to be the $L^2(K)$-orthogonal
projector to the space of affine functions defined over an element $K \in \cT_h$.
Similarly, for a facet $\gamma \subset \partial K$ we define
$\Pi_\gamma : L^2(\gamma) \rightarrow \mathbb{P}^1(\gamma)$ to be the $L^2(\gamma)$-orthogonal projector to the space of affine functions defined
over the facet $\gamma \subset \partial K$.
The following generalization of the corresponding result in \cite{robustaee:2010}
forms the basis of our approach:

%  $\displaystyle \Pi_K f \in P^1(K) :\quad  (f - \Pi_K f, \varphi)_K = 0
%   \quad \forall \varphi \in P^1(K) $

\begin{lemma}
\label{le:1}
Let $u \in V$ be the weak solution \ttg{eq:weakf} and
$u_h \in V$ be an arbitrary function.
Further let $\btau \in\Hdiv$ be such that
$\Pi_K f + \ddiv \btau = 0$
in those elements $K\in\cT_h$ where $\kappa_K = 0$
and
$\btau \cdot \bn = \Pi_\gamma \gN$
on facets $\gamma \subset \GammaN\cap\partial K$.
Then
$$
\trinorm{u - u_h}^2 \leq \sum\limits_{K\in\cT_h}
  \left[ \eta_K(\btau) + \osc_K(f) +
    \sum\limits_{\gamma \subset \GammaN\cap\partial K} \osc_{\gamma}(\gN)
  \right]^2
%\quad \forall \btau_K \in\b{\Sigma}_K
%\quad \forall \btau \in\Hdiv,
$$
where $\eta_K(\btau) \geq 0$, $\osc_K(f)$, and $\osc_{\gamma}(\gN)$
are defined by
\begin{align}
\label{eq:etaK}
  &\eta_K^2(\btau) = \left\{
  \begin{array}{l}
      \norm{\btau - \bnabla u_h}_K^2
          + \kappa_K^{-2} \norm{\Pi_K f - \kappa_K^2 u_h + \ddiv \btau }_K^2
       \quad \text{if } \kappa_K > 0,
   \\
     \norm{\btau - \bnabla u_h}_K^2
     \quad \text{if } \kappa_K = 0,
  \end{array} \right.
% \\
%   \eta_K(\btau) &= \left( \norm{\btau - \bnabla u_h}_K^2
%    + \min\left\{\frac{h_K}{\pi},\frac{1}{\kappa_K}\right\}^{2}
%       \norm{\Pi_K f - \kappa_K^2 u_h + \ddiv \btau }_K^2
%    \right)^{\frac12}
% \\ \label{eq:etaK}
%   &\quad + \min\{\CT^K,\CT^{K,\kappa_K} \} \norm{\Pi^K_\gamma \gN - \btau \cdot \bn}_{\GammaN\cap\partial K},
\\ \nonumber
&\osc_K(f) = \min \left\{ \frac{h_K}{\pi}, \frac{1}{\kappa_K} \right\}
   \norm{f - \Pi_K f}_K,
\\ \nonumber
&\osc_{\gamma}(\gN) =
   \min\{\CT,\oCT\} \norm{\gN - \Pi_\gamma \gN}_{\gamma}.
\end{align}
\end{lemma}
\begin{proof}
Let $v \in V$ be arbitrary.
Using the weak formulation \ttg{eq:weakf} for $u$,
the fact that the global forms $\cB$ and $\cF$ are sums of the local
forms $\cB_K$ and $\cF_K$, and the divergence theorem, we obtain
the following identity
\begin{multline}
\label{eq:errid}
  \cB(u - u_h,v) = \sum\limits_{K\in\cT_h} \cF_K(v) - \cB_K(u_h, v)
= \sum\limits_{K\in\cT_h} \left[
  \int_K (\btau - \bnabla u_h) \cdot \bnabla v\dx
\right.
\\
+ \int_K (\Pi_K f - \kappa_K^2 u_h + \ddiv \btau) v \dx
+ \sum\limits_{\gamma\subset\GammaN\cap\partial K}
  \int_{\gamma} (\Pi_\gamma \gN - \btau\cdot\bn) v \ds
\\
\left.
+ \int_K (f - \Pi_K f) v \dx
+ \sum\limits_{\gamma\subset\GammaN\cap\partial K} \int_{\gamma} (\gN - \Pi_\gamma \gN) v \ds
\right].
\end{multline}
Now we estimate the five integrals on the right-hand side of \ttg{eq:errid}.
The sum of the first two integrals is clearly bounded by
% $$
%   \left( \norm{\btau - \bnabla u_h}_K^2
%    + \kappa_K^{-2} \norm{\Pi_K f - \kappa_K^2 u_h + \ddiv \btau }_K^2
%    \right)^{\frac12} \trinorm{v}_K.
% $$
$
  \eta_K(\btau) \trinorm{v}_K
$
for both $\kappa_K > 0$ and $\kappa_K = 0$.
The third integral on the right-hand side of \ttg{eq:errid} vanishes
since $\btau\cdot\bn = \Pi_\gamma \gN$ on $\GammaN$.

The fourth integral can be estimated as
$$
  \int_K (f - \Pi_K f) v \dx \leq \min \left\{ \frac{h_K}{\pi}, \frac{1}{\kappa_K} \right\}
   \norm{f - \Pi_K f}_K \trinorm{v}_K
 = \osc_K(f) \trinorm{v}_K,
$$
where the constant $h_K/\pi$ comes from the Poincar\'e inequality \cite{PayWei:1960}
and $1/\kappa_K$ comes from the inequality $\norm{v}_K \leq \kappa_K^{-1}\trinorm{v}_K$,
see \cite[p.~228]{robustaee:2010} for details.
The last integral in \ttg{eq:errid} can be bounded in the following two ways:
\begin{align*}
\int_{\gamma} (\gN - \Pi_\gamma \gN) v \ds
&\leq \norm{\gN - \Pi_\gamma \gN}_{\gamma} \norm{ v }_{\gamma},
\\
  \int_{\gamma} (\gN - \Pi_\gamma \gN) v \ds
&=\int_{\gamma} (\gN - \Pi_\gamma \gN) (v-\bar v_\gamma) \ds
\leq \norm{\gN - \Pi_\gamma \gN}_{\gamma}
     \norm{ v - \bar v_\gamma }_{\gamma},
\end{align*}
where $\bar v_\gamma = |\gamma|^{-1} \int_{\gamma} v \ds$.
Employing trace inequalities \ttg{eq:trace2} and \ttg{eq:trace1} we end up
with the estimate
$$
\int_{\gamma} (\gN - \Pi_\gamma \gN) v \ds \leq \osc_{\gamma}(\gN) \trinorm{v}_K.
$$

Hence,
$$
 \cB(u - u_h,v) \leq \sum\limits_{K\in\cT_h}
\left[ \eta_K(\btau) + \osc_K(f) +
  \sum\limits_{\gamma\subset\GammaN\cap\partial K} \osc_{\gamma}(\gN) \right] \trinorm{v}_K.
$$
%
% The facts that $\norm{v}_K \leq \kappa_K^{-1}\trinorm{v}_K$,
% $\int_K (f - \Pi_K f) c \dx = 0$ for any constant $c$, and Poincar\'e inequality
% $\inf_{c\in\R} \norm{v - c}_K \leq h_K \pi^{-1} \norm{\bnabla v}_K$
% \cite{PayWei:1960}, enable to bound the last integral in \ttg{eq:errid} as
% $$
%   \int_K (f - f_K) v \dx \leq \min \left\{ \frac{h_K}{\pi}, \frac{1}{\kappa_K} \right\}
%    \norm{f - \Pi_K f}_K \trinorm{v}_K.
% $$
% The sum of the first two integrals on the right-hand side of \ttg{eq:errid}
% is readily bounded by $\eta_K(\btau)\trinorm{v}_K$ and hence
% $$
%  \cB(u - u_h,v) \leq \sum\limits_{K\in\cT_h}
% \left[ \eta_K(\btau) + \min \left\{ \frac{h_K}{\pi}, \frac{1}{\kappa_K} \right\}
%    \norm{f - \Pi_K f}_K  \right] \trinorm{v}_K.
% $$
%
The Cauchy-Schwarz inequality and substitution $v=u-u_h$ finishes the proof.
\end{proof}

The vector field $\btau \in \Hdiv$ is referred to as
a \emph{flux reconstruction} and its specific choice is crucial
for the efficiency and robustness of the resulting error estimators.
We reconstruct the flux $\btau\in\Hdiv$ in two steps.
Firstly, we find boundary fluxes $g_K$ satisfying the following conditions:
\begin{alignat}{2}
\label{eq:gKlin}
g_K|_\gamma &\in \mathbb{P}^1(\gamma)
&\quad&\text{for all facets } \gamma \subset \partial K,
\\ \label{eq:gKNeu}
g_K &= \Pi_\gamma \gN
&\quad&\text{for all facets } \gamma \subset \GammaN\cap\partial K,
\\
\label{eq:gKconsist}
  g_K + g_{K'} &= 0
  &\quad&\text{on facets } \gamma = \partial K \cap \partial K'.
\end{alignat}
% where $\theta_n \in \mathbb{P}^1(K)$ are the standard affined finite element
% shape functions in the simplex $K$, i.e., the are equal to one at one of the
% vertices and vanish at all the other vertices.
Secondly, we locally reconstruct vector fields $\btau_K \in \Hdiv[K]$ satisfying
boundary conditions $\btau_K \cdot \bn_K = g_K$ on $\partial K$.
The values of $\btau$ in the interior of the elements will be presented
in detail below. Irrespective of this, the
resulting vector field $\btau$ is defined elementwise by $\btau|_K = \btau_K$
for all $K\in \cT_h$, so that $\btau\in\Hdiv$ due to the
consistency condition \ttg{eq:gKconsist}.

Conditions \ttg{eq:gKlin}--\ttg{eq:gKconsist} do not determine a unique set of
fluxes. The specific choice of fluxes satisfying \ttg{eq:gKlin}--\ttg{eq:gKconsist}
will be crucial to the robustness of the associated estimator.
We say that boundary fluxes $g_K$ are equilibrated with respect to linear
functions if the following condition holds:
\begin{equation}
\label{eq:gKequilib}
  \int_K f \theta \dx %\cF_K(\theta)
  - \cB_K(u_h,\theta)
  %\int_K (f \theta - \bnabla u_h \cdot \bnabla \theta - \kappa_K^2 u_h \theta) \dx
  + \int_{\partial K} g_K \theta \ds
 = 0 \quad\forall \theta \in \mathbb{P}^1(K).  %n=1,2,\dots,d,
\end{equation}
or, equally well,
\begin{equation*}
%\label{eq:tauequilib}
  \int_K (\btau - \bnabla u_h) \cdot \bnabla \theta \dx
  + \int_K (f - \kappa_K^2 u_h + \ddiv\btau) \theta \dx = 0
\quad\forall \theta\in \mathbb{P}^1(K).
\end{equation*}

Fluxes satisfying the
equilibration \ttg{eq:gKequilib} yield accurate error bounds for small $\kappa$,
but are not robust for large values of $\kappa$ \cite{AinBab:1999}.
Therefore, we will follow the approach from \cite{AinBab:1999}
for large values of $\kappa$.

\subsection{Robust equilibration of boundary fluxes}
\label{se:robequilib}
A detailed algorithm for the construction of boundary fluxes $g_K$
satisfying conditions \ttg{eq:gKlin}--\ttg{eq:gKequilib}
can be found in \cite{AinOde:2000}.
A modification of this approach that is robust for large values
of $\kappa$ is described in \cite{AinBab:1999} and
\cite{robustaee:2010} and we will briefly recall it here.

The idea is to replace the affine functions in \ttg{eq:gKequilib} by
their approximate minimum energy extensions. Clearly, it suffices
to satisfy condition \ttg{eq:gKequilib} for the barycentric coordinates
$\theta_n$, $n=1,2,\dots,d+1$, in $K$.
The approximate minimum energy extensions $\theta^*_n$
of $\theta_n$ are defined in \cite{AinBab:1999} for $d=1$, $2$, and $3$ dimensions.
Here, we define them for general $d$-dimensional simplices.

Consider a simplex $K$ with vertices $\bfx_1, \dots, \bfx_{d+1}$
and facets $\gamma_1, \dots, \gamma_{d+1}$ opposite to these vertices.
The standard basis functions $\theta_n$ are determined by the conditions
$\theta_n(\bfx_m) = \delta_{nm}$, $n,m \in \{1,2,\dots,d+1\}$.
For each $n = 1,2,\dots,d+1$, define approximate minimum energy extension
$\theta^*_n$ as follows.
If $\kappa_K \rho_K \leq 1$ then $\theta^*_n = \theta_n$.
If $\kappa_K \rho_K > 1$ then define a point $\bfx_P$
by its barycentric coordinates
$\lambda_i(\bfx_P) = \delta$ for $i\neq n$
and $\lambda_n(\bfx_P) = 1-d\delta$
with $\delta = \min\{1,1/(\kappa_K\rho_K)\}/d$,
and consider a submesh in $K$ created by simplices
$K_i = \overline{\gamma_i \bfx_P}$, $i=1,2,\dots,d+1$.
The approximate minimum energy extension $\theta^*_n$ is then defined
as a piecewise affine function with respect to this submesh such that
$\theta^*_n(\bfx_n) = 1$, $\theta^*_n(\bfx_P) = 0$, and
$\theta^*_n(\bfx_i) = 0$ for all $i\neq n$.
A two-dimensional illustration of functions $\theta^*_n$ is provided in
Figure~\ref{fi:theta}.

\begin{figure}
\begin{center}
\newlength{\ww}
\setlength{\ww}{1.33mm}
\includegraphics[width=40\ww]{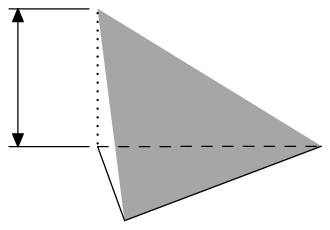}%
\makebox[0pt][l]{\hspace{-20\ww}\raisebox{23\ww}[0pt][0pt]{$\theta_3^*$}}%
\makebox[0pt][l]{\hspace{-41\ww}\raisebox{17\ww}[0pt][0pt]{$1$}}%
\makebox[0pt][l]{\hspace{-26\ww}\raisebox{-2\ww}[0pt][0pt]{$\bfx_1$}}%
\makebox[0pt][l]{\hspace{0.5\ww}\raisebox{9\ww}[0pt][0pt]{$\bfx_2$}}%
\makebox[0pt][l]{\hspace{-32\ww}\raisebox{7\ww}[0pt][0pt]{$\bfx_3$}}%
\qquad\qquad
\includegraphics[width=40\ww]{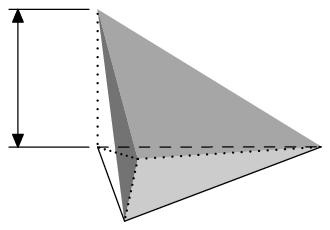}%
\makebox[0pt][l]{\hspace{-20\ww}\raisebox{23\ww}[0pt][0pt]{$\theta_3^*$}}%
\makebox[0pt][l]{\hspace{-41\ww}\raisebox{17\ww}[0pt][0pt]{$1$}}%
\makebox[0pt][l]{\hspace{-26\ww}\raisebox{-2\ww}[0pt][0pt]{$\bfx_1$}}%
\makebox[0pt][l]{\hspace{0.5\ww}\raisebox{9\ww}[0pt][0pt]{$\bfx_2$}}%
\makebox[0pt][l]{\hspace{-32\ww}\raisebox{7\ww}[0pt][0pt]{$\bfx_3$}}%
\makebox[0pt][l]{\hspace{-23.5\ww}\raisebox{5.5\ww}[0pt][0pt]{$\bfx_P$}}%
\end{center}
\caption{\label{fi:theta}
A graph of the approximate minimum energy extension $\theta^*_n$ for
$\kappa_K \rho_K \leq 1$ (left) and for $\kappa_K \rho_K > 1$ (right).
}
\end{figure}

It is easy to verify that $\theta^*_n$, $n\in\{1,2,\dots,d+1\}$, satisfy
\begin{itemize}
\itemsep=0pt
\item $\theta^*_n = \theta_n$ on the boundary $\partial K$;
\item if $\kappa_K \rho_K \leq 1$ then $\theta^*_n=\theta_n$ on $K$;
\item $C_1 h_K^{d-1} \min\{h_K, 1/\kappa_K\} \leq \norm{\theta^*_n}_K^2 \leq C_2 h_K^{d-1} \min\{h_K, 1/\kappa_K\}$;
\item $C_1 h_K^{d-1} \min\{h_K, 1/\kappa_K \}^{-1} \leq \norm{\bnabla \theta^*_n}_K^2 \leq C_2 h_K^{d-1} \min\{h_K, 1/\kappa_K \}^{-1}$.
\end{itemize}
These key features of the approximate minimum energy extensions
$\theta^*_n$ were identified already in \cite{AinBab:1999} and
are crucial for the robustness of the resulting fluxes.

The robust flux reconstruction is obtained in \cite{AinBab:1999} by replacing
functions $\theta_n$ by $\theta^*_n$ in \ttg{eq:gKequilib}
and finding the least squares minimizer of the system. This approach must
be modified to deal with case of variable $\kappa$ considered here.

First, we define
$$
  \varepsilon_K(\theta) = \cF_K(\theta) %\int_K f \theta_n \dx
  - \cB_K(u_h,\theta)
  + \int_{\partial K\setminus\GammaN} g_K \theta \ds
$$
for any $\theta \in H^1(K)$.
The robust equlibration procedure requires
\begin{equation}
\label{eq:equilib1}
%     \cF_K(\theta_n) %\int_K f \theta_n \dx
%   - \cB_K(u_h,\theta_n)
%   + \int_{\partial K\setminus\GammaN} g_K \theta_n \ds
%  = 0
\varepsilon_K(\theta_n) = 0
\quad\forall K\in\cT_h, \text{ where } \kappa_K \rho_K \leq 1,
\text{ and } n=1,2,\dots,d+1.
%\quad\forall K\in\cT_h^\mathrm{exact},\ n=1,2,\dots,d+1,
\end{equation}
%for all elements $K\in\cT_h$, where $\kappa_K \rho_K \leq 1$,
%and for all $n=1,2,\dots,d+1$.
For the other elements, %$K\in\cT_h$, where $\kappa_K \rho_K > 1$,
we impose a similar condition in a least-squares sense.
%In particular, fluxes $g_K$ are required to minimize
We obtain a constrained least-squares problem
to find $g_K$ satisfying \ttg{eq:gKlin}--\ttg{eq:gKconsist},
equality constraints \ttg{eq:equilib1} and minimizing
\begin{equation}
\label{eq:equilib2}
\sum\limits_{K\in\cT_h : \kappa_K \rho_K > 1}
\sum\limits_{n=1}^{d+1} \varepsilon_K(\theta^*_n)^2.
\end{equation}
% \begin{equation}
% \label{eq:equilib2}
%     \cF_K(\theta^*_n) %\int_K f \theta^*_n \dx
%   - \cB_K(u_h,\theta^*_n)
%   + \int_{\partial K\setminus\GammaN} g_K \theta^*_n \ds
%  \approx 0 \quad\forall n=1,2,\dots,d+1.
% \end{equation}
% where the symbol $\approx$ stands for the least-squares fit and
% \begin{align*}
% \cT_h^\mathrm{exact} &= \{ K \in \cT_h : \kappa_K \rho_K \leq 1
% %\text{ and } \bx_n \not\in\oGammaD\ \forall n=1,2,\dots,d
% \},
% %\bx_n \in \cN(K) \},
% \\
% \cT_h^\mathrm{approx} &= \cT_h \setminus \cT_h^\mathrm{exact}.
% \end{align*}
This global constrained least-squares problem
can be transformed into a series of small constrained least-squares
problems on patches of elements corresponding to vertices of $\cT_h$
as follows.

We define the set of vertices $\cN(\gamma)$ of a facet $\gamma$
of a simplex $K$ and functions $\psi_\gamma^m \in \mathbb{P}^1(\gamma)$
%\in \operatorname{span} \{\theta_n : \bx_n \in\cN(\gamma) \}
satisfying
$
  \int_\gamma \psi_\gamma^m \theta_n \ds = \delta_{mn}.
$
Further, we consider a fixed orientation $\sigma_{K,\gamma}$
of facets $\gamma$ of simplices $K \in \cT_h$.
The orientation $\sigma_{K,\gamma}$ is either $1$ or $-1$ and satisfies
$$
  \sigma_{K,\gamma} + \sigma_{K',\gamma} = 0
  \quad\text{on } \gamma=\partial K \cap \partial K'.
$$
Finally, we introduce the average and the jump flux across a common
facet of two neighbouring simplices $K$ and $K'$ as
$$
  \left\langle \frac{\partial u_h}{\partial\bn_K} \right\rangle
 = \frac12 \bn_K \cdot \left( \bnabla u_h|_K + \bnabla u_h|_{K'} \right)
\quad\text{and}\quad
  \left[ \frac{\partial u_h}{\partial\bn_K} \right]
 = \bn_K \cdot \left( \bnabla u_h|_K - \bnabla u_h|_{K'} \right).
$$
On the boundary $\partial\Omega$ we set
$\left\langle \partial u_h/\partial\bn_K \right\rangle =
\partial u_h/\partial\bn_K$ and $[\partial u_h/\partial\bn_K] = 0$.
The boundary flux $g_K$ on a facet $\gamma$ of a simplex $K$ is then
defined in the form
\begin{equation}
\label{eq:gKansatz}
  g_K = \left\langle \frac{\partial u_h}{\partial\bn_K} \right\rangle
  + \sigma_{K,\gamma} \sum\limits_{_m\in\cN(\gamma)} \alpha_\gamma^m \psi_\gamma^m.
\end{equation}
Notice that this construction of $g_K$ immediately guarantees
the consistency condition \ttg{eq:gKconsist}.
Furthermore, if $\gamma \subset \GammaN$ then the coefficients
$\alpha_\gamma^m$ are uniquely determined by \ttg{eq:gKNeu}.

Using \ttg{eq:gKansatz}, we can readily express $\varepsilon_K(\theta)$ as
\begin{equation}
\label{eq:loceps}
  \varepsilon_K(\theta) = D_K(\theta) +
  \sum\limits_{\gamma: \gamma \subset \partial K\setminus\GammaN, \gamma\ni\bx_n}
  \sigma_{K,\gamma} \alpha^n_\gamma,
\end{equation}
where
$$
  D_K(\theta) = \cF_K(\theta)
  - \cB_K(u_h,\theta)
  + \int_{\partial K\setminus\GammaN}
    \left\langle \frac{\partial u_h}{\partial\bn_K} \right\rangle
    \theta \ds.
$$
Note that the summation in \ttg{eq:loceps} is performed over those
facets $\gamma \subset \partial K \setminus\GammaN$ that contain
the vertex $\bx_n$.

Thus, the global constrained least-squares problem
\ttg{eq:equilib1}--\ttg{eq:equilib2} transforms into
the following small local constrained least-squares problems on patches of elements.
For all vertices $\bx_n$ in the partition $\cT_h$, we define
$\omega(\bx_n) = \{K \in \cT_h: \bx_n \in \overline K\}$ to be the set of
elements sharing the vertex $\bx_n$ and seek coefficients $\alpha^n_\gamma$
for $\gamma \not\subset \GammaN$, $\gamma\ni\bx_n$
satisfying equality constraints
\begin{equation}
\label{eq:loceq1}
  \varepsilon_K(\theta_n) = 0
  \quad\forall K\in \omega(\bx_n), \text{where } \kappa_K \rho_K \leq 1,
\end{equation}
and minimizing
\begin{equation}  
\label{eq:loceq2}
  %\min\limits_{\alpha^n_\gamma : \gamma \subset \partial K\setminus\GammaN, \gamma\ni\bx_n}
  \sum\limits_{K\in\omega(\bx_n) : \kappa_K \rho_K > 1} \varepsilon_K(\theta^*_n)^2.
%\quad\forall K\in \omega(\bx_n), \text{where } \kappa_K \rho_K > 1,
%K \subset \operatorname{supp} \theta_n
\end{equation}
% \begin{align}
% \label{eq:loceq1}
%   \sum\limits_{\gamma: \gamma \subset \partial K\setminus\GammaN, \gamma\ni\bx_n}
%   \sigma_{K,\gamma} \alpha^n_\gamma &= -D_K(\theta_n)
%   \quad\forall K\in \omega(\bx_n), \text{where } \kappa_K \rho_K \leq 1,
% \\ \label{eq:loceq2}
%   \sum\limits_{\gamma: \gamma \subset \partial K\setminus\GammaN, \gamma\ni\bx_n}
%   \sigma_{K,\gamma} \alpha^n_\gamma &\approx -D_K(\theta^*_n)
% \quad\forall K\in \omega(\bx_n), \text{where } \kappa_K \rho_K > 1,
% %K \subset \operatorname{supp} \theta_n
% \end{align}
Note that $\varepsilon_K$ in \ttg{eq:loceq1} and \ttg{eq:loceq2} is considered
in the form \ttg{eq:loceps}.

Observe that the system \ttg{eq:loceq1}--\ttg{eq:loceq2} is always solvable.
It was shown in \cite{AinBab:1999} and \cite{AinOde:2000} that
if $\kappa_K\rho_K \leq 1$ for all elements $K \in \omega(\bx_n)$ then
the linear system \ttg{eq:loceq1} is always solvable.
Trivially, removing some (or all) equality constraints
and replacing them by the requirement of least-squares fit \ttg{eq:loceq2},
preserves the solvability of the remaining system of linear
constraints \ttg{eq:loceq1}. In case the constrained least-squared problem \ttg{eq:loceq1}--\ttg{eq:loceq2} does not have a unique solution,
we consider the solution with the smallest $L^2$ norm.

To summarize, we require the satisfaction of the exact equilibration condition
\ttg{eq:loceq1} for those elements where the coefficient $\kappa_K$ is small.
For the other elements we
mimic the equilibration by the least-squares fit \ttg{eq:loceq2}.
The resulting constrained least-squares problem \ttg{eq:loceq1}--\ttg{eq:loceq2}
is always solvable and its solution depends continuously on the data.

\subsection{Auxiliary results}
In this section we recall several estimates from \cite{AinBab:1999}
and extend them to include the case of piecewise constant $\kappa$ and to Neumann
boundary conditions.
Let us note that assumptions \ttg{eq:kappacond1}--\ttg{eq:kappacond2} and consequently \ttg{eq:mincond}
make this extension straightforward.
Lemma~5(2) from \cite{AinBab:1999} says that
if $\gamma$ is an interior facet (i.e.\ shared by two elements)
then
\begin{equation}
\label{eq:jumpest}
\norm{ \left[ \frac{\partial u_h}{\partial\bn_K} \right] }_\gamma
  \leq C \left[
     \min\{h_\gamma,\kappa_K^{-1} \}^{-\frac12}
     \trinorm{u - u_h}_{\widetilde\gamma}
   +
     \min\{h_\gamma,\kappa_K^{-1} \}^{\frac12}
     \norm{f - \Pi f}_{\widetilde\gamma}
 \right],
\end{equation}
where
$\widetilde\gamma$ is the pair of elements sharing the facet $\gamma$
and $\Pi f$ is defined piecewise by $(\Pi f)|_K = \Pi_K f$ for
all $K\in\cT_h$.
Notice that thanks to the assumption \ttg{eq:mincond}
estimate \ttg{eq:jumpest} holds for any $K \in \widetilde\gamma$.

Further, Lemma~6 from \cite{AinBab:1999} provides the estimate
\begin{equation}
\label{eq:avgfluxest}
  \norm{g_K-\left\langle\frac{\partial u_h}{\partial\bn_K}\right\rangle}_{\gamma}
  \leq C \left[
    \min\{h_\gamma,\kappa_K^{-1} \}^{-\frac12}
    \trinorm{u - u_h}_{\widetilde K}
   +
    \min\{h_\gamma,\kappa_K^{-1} \}^{\frac12}
    \norm{f - \Pi f}_{\widetilde K}
 \right],
\end{equation}
for all $K\in \cT_h$, where $\gamma$ is a facet of $K$ that is either interior
or it lies on the Dirichlet boundary $\GammaD$.
The patch of elements $\widetilde K$ was defined in \ttg{eq:patchK}
and condition \ttg{eq:mincond} is needed to generalize the proof from \cite{AinBab:1999}
to the case of piecewise constant $\kappa$.
The final estimate on page 343 in \cite{AinBab:1999} states that
\begin{equation}
\label{eq:Pirest}
  \norm{\Pi_K r_h}_K
  \leq C \left[ 
   \min\{h_K,\kappa_K^{-1}\}
   \trinorm{u - u_h}_K + \norm{f - \Pi_K f}_K \right],
\end{equation}
for all elements $K$ in $\cT_h$.
Here, $r_h = f - \kappa_K^2 u_h + \Delta u_h$ stands for the residual on $K$.
This bound is local and independent of values of $\kappa$ in the other elements
and therefore applies to the case considered here.

We emphasize that estimates \ttg{eq:jumpest}--\ttg{eq:Pirest} are proved
in \cite{AinBab:1999} for the case of pure Dirichlet boundary conditions
and constant coefficient $\kappa$.
However, their proofs remain valid even in the presence of Neumann boundary
conditions and due to the condition \ttg{eq:mincond} also for piecewise
constant $\kappa$.
Nevertheless, estimate \ttg{eq:avgfluxest}
is not valid for a facet on the Neumann boundary.
We use a slight modification of the proof of Lemma~5(2)
from \cite{AinBab:1999} and derive the estimate
\begin{multline}
\label{eq:Neufluxes}
  \norm{g_K - \bnabla u_h|_K \cdot \bn_K}_{\gamma}
  \leq C \left[
    \min\{h_K,\kappa_K^{-1}\}^{-\frac12} 
    \trinorm{u - u_h}_K
\right.
\\ \left.
   +\min\{h_K,\kappa_K^{-1}\}^{\frac12} 
    \norm{f - \Pi_K f}_K
   + \norm{\gN - \Pi_\gamma \gN}_{\gamma}
 \right]
\end{multline}
for those facets $\gamma$ of $K$ located on the Neumann boundary $\GammaN$.
Thus, defining $R = g_K - \bnabla u_h|_K \cdot \bn_K$,
using \ttg{eq:jumpest}, \ttg{eq:avgfluxest}, \ttg{eq:Neufluxes}
and the fact that
$$
  R %= g_K - \bnabla u_h|_K \cdot \bn_K
= g_K - \left\langle\frac{\partial u_h}{\partial\bn_K}\right\rangle
   - \frac12 \left[ \frac{\partial u_h}{\partial\bn_K} \right]
$$
we easily derive the estimate
\begin{multline}
\label{eq:normRest}
 \norm{R}_{\partial K} \leq C \left(
     \min\{h_K,\kappa_K^{-1}\}^{-\frac12}
     \trinorm{u - u_h}_{\widetilde K}
\right.
\\ \left.
   +
     \min\{h_K,\kappa_K^{-1}\}^{\frac12}
     \norm{f - \Pi f}_{\widetilde K}
   + \norm{\gN - \Pi^K_\gamma \gN}_{\GammaN\cap\partial K}
\right)
\end{multline}
for all $K\in \cT_h$. In view of notation \ttg{eq:kappa0}
and thanks to the assumption \ttg{eq:kappacond2} the above estimates hold
even if $\kappa_K = 0$.

% \begin{equation}
% \label{eq:normRest}
%  \kappa^{-1/2} \norm{R}_{\partial K} \leq C \left(
%   \trinorm{u - u_h}_{\widetilde K} + \kappa^{-1} \norm{f - \Pi f}_{\widetilde K}
% \right).
% \end{equation}

\subsection{Flux reconstruction \#1}
For each simplex $K\in\cT_h$ on which $\kappa_K\rho_K\leq 1$,
we use a reconstruction of the form
\begin{equation}
\label{eq:tauK1}
  \btau_K^{(1)} = \bnabla u_h|_K + \btauKL + \btauKQ.
\end{equation}
The vector field $\btauKL$ is defined as
\begin{equation}
\label{eq:tauL}
  \btauKL = -\sum\limits_{n=1}^{d+1} \lambda_n
     \sum\limits_{\begin{subarray}{c} m=1\\ m\neq n \end{subarray}}^{d+1}
     R_{|\gamma_m}(\bfx_n) \; |\bnabla \lambda_m| \; \bt_{nm},
\end{equation}
where
% $$
%   \bc_n = %-\sum\limits_{m\in\cN(K), m\neq n}
%      -\sum\limits_{\begin{subarray}{c} m=1\\ m\neq n \end{subarray}}^{d+1}
%      R_{|\gamma_m}(\bfx_n) \; |\bnabla \lambda_m| \; \bt_{nm}
% $$
$\bfx_n$, $n=1,2,\dots,d+1$, stand for vertices of $K$,
$\gamma_n$ are the facets opposite to the vertices $\bfx_n$,
$\lambda_n$ are the corresponding barycentric coordinates,
$\bt_{mn} = \bfx_n - \bfx_m$ denote the edge-vectors from $\bfx_m$ to $\bfx_n$,
and function $R = g_K - \bnabla u_h|_K \cdot \bn_K$ is affine on each facet of $K$.
The quadratic vector field $\btauKQ$ is given by
\begin{equation}
\label{eq:tauQ}
  \btauKQ = \frac{1}{d+1} %\sum\limits_{m,n\in\cN(K), m > n}
   %\sum\limits_{\begin{subarray}{c} m,n=1\\ m > n \end{subarray}}^{d+1}
   \sum\limits_{n=1}^{d+1}
   \sum\limits_{\begin{subarray}{c} m=2\\ m > n \end{subarray}}^{d+1}
  \lambda_m \lambda_n \bt_{mn} \bt_{mn}^T \bnabla r(\overline\bfx_K)
\end{equation}
where $r = \Pi_K f - \kappa_K^2 u_h$ is affine on $K$,
and $\overline\bfx_K$ denotes the centroid of simplex $K$.

It can be easily shown that $\btauKL \cdot \bn_K = R$ 
and $\btauKQ \cdot \bn_K = 0$ on each facet of $K$.
Indeed, if we denote the outward normal unit vector to the facet $\gamma_k$
by $\bn_k$ then the following identity holds
\begin{align*}
  \btauKL \cdot \bn_k|_{\gamma_k}
%  = -\sum\limits_{n=1}^{d+1}
%      %\sum\limits_{m\in\cN(K), m\neq n}
%      \sum\limits_{\begin{subarray}{c} m=1\\ m\neq n \end{subarray}}^{d+1}
%      \lambda_n R_{|\gamma_m}(\bfx_n) \; |\bnabla \lambda_m| \; \bt_{nm}\cdot\bn_k
 &= %-\sum\limits_{n\in\cN(K), n\neq k}
    -\sum\limits_{\begin{subarray}{c} n=1\\ n\neq k \end{subarray}}^{d+1}
     \lambda_n R_{|\gamma_k}(\bfx_n) \; |\bnabla \lambda_k| \; \bt_{nk}\cdot\bn_k
\\
 &= \sum\limits_{\begin{subarray}{c} n=1\\ n\neq k \end{subarray}}^{d+1}
     \lambda_n R_{|\gamma_k}(\bfx_n) \; \bt_{nk}\cdot\bnabla\lambda_k
 = \sum\limits_{\begin{subarray}{c} n=1\\ n\neq k \end{subarray}}^{d+1}
     \lambda_n R_{|\gamma_k}(\bfx_n)
 = R_{|\gamma_k},
\end{align*}
where we use the facts that: $\lambda_k|_{\gamma_k} = 0$;
if $n\neq k$ then $\bt_{nm}\cdot\bn_k = 0$ for $m\neq k$;
%$\bt_{mn}\cdot\bn_k$ is nonzero for $m=k$ and $n=k$, only,
$\bn_k = -\bnabla \lambda_k/|\bnabla \lambda_k|$;
and $\bt_{nk} \cdot \bnabla\lambda_k = 1$.
Similarly, we show that
$$
 \btauKQ \cdot \bn_k|_{\gamma_k}
= \frac{1}{d+1} %\sum\limits_{m,n\in\cN(K), m>n}
  %\sum\limits_{\begin{subarray}{c} m,n=1\\ m > n \end{subarray}}^{d+1}
   \sum\limits_{n=1}^{d+1}
   \sum\limits_{\begin{subarray}{c} m=2\\ m > n \end{subarray}}^{d+1}
  \lambda_m|_{\gamma_k} \lambda_n|_{\gamma_k} (\bt_{mn} \cdot \bnabla r(\overline\bfx_K)) (\bt_{mn} \cdot \bn_k)
= 0,
$$
because
if $n=k$ then $\lambda_n|_{\gamma_k} = 0$ and
if $n\neq k$ then $m\neq k$ and $\bt_{mn} \cdot \bn_k = 0$.

%Thus, we verified that 

%In addition the flux $\btau_K^{(1)}$ is constructed in such a way that
%the first term appearing in the definition \ttg{eq:etaK} of $\eta_K$
%vanishes. We formulate this statement as the following lemma.

\begin{lemma}\label{le:divtau1}
Let $K \in \cT_h$ then the vector field $\btau_K^{(1)}$ defined by
\ttg{eq:tauK1}, \ttg{eq:tauL}, and \ttg{eq:tauQ} satisfies
$\btau_K^{(1)} \cdot \bn_K = g_K$ on all facets of $K$
and,
if $\kappa_K \rho_K \leq 1$, then 
$$
  \Pi_K f - \kappa_K^2u_h + \ddiv \btau_K^{(1)} = 0
  \quad\text{in }K.
$$
\end{lemma}
\begin{proof}
The first assertion is a consequence of the foregoing arguments.
Suppose $\kappa_K \rho_K \leq 1$, then
since $\btauKL$ has constant divergence over the element $K$
and $\bnabla u_h|_K$ has vanishing divergence over $K$, we
have
$$
\ddiv \btauKL = \frac{1}{|K|} \int_{\partial K} \btauKL\cdot\bn_K \ds
 = \frac{1}{|K|} \int_{\partial K} g_K \ds.
$$
%The fact \ttg{eq:shaperegrhoK} and assumption $\kappa_K\rho_K \leq 1/\cC_0$
Since $\kappa_K \rho_K \leq 1$, the exact equilibration condition
\ttg{eq:equilib1} is satisfied in $K$ for $n=1,2,\dots,d+1$ and,
consequently, using the fact that $\sum_{n=1}^{d+1} \theta_n = 1$ on $K$
in \ttg{eq:equilib1} we end up with equality
$$
  \int_{\partial K} g_K \ds = -\int_K (f - \kappa_K^2 u_h) \dx
  = -\int_K (\Pi_K f - \kappa_K^2 u_h) \dx.
$$
Observing that $\ddiv (\lambda_m \lambda_n \bt_{mn}) = \lambda_m - \lambda_n$
and that
$$
 %\sum\limits_{m,n\in\cN(K), m>n}
 %\sum\limits_{\begin{subarray}{c} m,n=1\\ m > n \end{subarray}}^{d+1}
 \sum\limits_{n=1}^{d+1}
 \sum\limits_{\begin{subarray}{c} m=2\\ m > n \end{subarray}}^{d+1}
 \left(\lambda_m(\bfx) - \lambda_n(\bfx)\right)(\bfx_n - \bfx_m)
 = -(d+1)(\bfx - \overline\bfx_K),
$$
where $\bfx = \sum_{m=1}^{d+1} \bfx_m \lambda_m(\bfx)$ and
$\overline\bfx_K = \sum_{m=1}^{d+1} \bfx_m / (d+1)$,
we can compute the divergence of $\btauKQ$ as
\begin{multline*}
  \ddiv\btauKQ(\bfx)
% &= \frac{1}{d+1} %\sum\limits_{m,n\in\cN(K), m>n}
%    \sum\limits_{\begin{subarray}{c} m,n=1\\ m > n \end{subarray}}^{d+1}
%   \ddiv(\lambda_m \lambda_n \bt_{mn})
%   \bt_{mn} \cdot \bnabla r(\overline\bfx_K))
% \\
= \frac{1}{d+1} %\sum\limits_{m,n\in\cN(K), m>n}
  %\sum\limits_{\begin{subarray}{c} m,n=1\\ m > n \end{subarray}}^{d+1}
  \sum\limits_{n=1}^{d+1}
  \sum\limits_{\begin{subarray}{c} m=2\\ m > n \end{subarray}}^{d+1}
  (\lambda_m(\bfx) - \lambda_n(\bfx))(\bfx_n - \bfx_m)
  \cdot \bnabla r(\overline\bfx_K)
\\
= \left(\overline\bfx_K - \bfx \right)\cdot \bnabla r(\overline\bfx_K).
\end{multline*}
Using the fact that $r = \Pi_K f - \kappa_K^2 u_h$ is affine
and the centroid quadrature
rule for simplices that is exact for all linear functions, we obtain
$$
  \ddiv\btauKQ(\bfx) = - r(\bfx) + r(\overline\bfx_K)
  = - r(\bfx) + \frac{1}{|K|}\int_K r \dx.
$$
The statement of the lemma follows by summing the above equations.
\end{proof}

The next result shows that $\btau_K^{(1)}$ gives an efficient estimate
of the local error in element $K$:

\begin{lemma}\label{le:tauK1}
%Let $\kappa_K h_K \leq 1$ ???
If $K\in\cT_h$ is such that $\kappa_K \rho_K \leq 1$
then
$$
  \eta_K\bigl(\btau_K^{(1)}\bigr) \leq C \left(
    \trinorm{u-u_h}_{\widetilde K}
  + h_K \norm{f-\Pi f}_{\widetilde K}
  + h_K^{1/2} \norm{\gN - \Pi^K_\gamma \gN}_{\GammaN\cap\partial K}
  \right).
$$
% where the final term is defined by the relation
% $$
%   \norm{f-\Pi f}_{\widetilde K}^2
% = \sum\limits_{K \subset \widetilde K} \norm{f-\Pi_K f}_K^2.
% $$
\end{lemma}
\begin{proof}
Let
$$
   \bc_n = \sum\limits_{\begin{subarray}{c} m=1\\ m\neq n \end{subarray}}^{d+1}
     R_{|\gamma_m}(\bfx_n) \; |\bnabla \lambda_m| \; \bt_{nm},
$$
then $\btauKL = -\sum_{n=1}^{d+1} \lambda_n \bc_n$ and we have
$$
  |\bc_n| \leq \sum\limits_{\begin{subarray}{c} m=1\\ m\neq n \end{subarray}}^{d+1}
     |R_{|\gamma_m}(\bfx_n)| \; |\bnabla \lambda_m| \; |\bt_{nm}|
\leq C \sum\limits_{\begin{subarray}{c} m=1\\ m\neq n \end{subarray}}^{d+1}
   |R_{|\gamma_m}(\bfx_n)|,
$$
because $d |K| |\bnabla\lambda_m| = |\gamma_m|$, $|\bt_{nm}| \leq h_K$,
and due to the shape regularity assumption \ttg{eq:shapereg}.
Further, thanks to the linearity of $R$,
$$
  \sum\limits_{\begin{subarray}{c} n=1\\ n\neq m \end{subarray}}^{d+1}
  | R_{|\gamma_m}(\bfx_n) |^2 \leq C \frac{1}{|\gamma_m|} \norm{R}_{\gamma_m}^2.
$$
We utilize these results to bound
\begin{equation}
\label{eq:tauKLest1}
\norm{\btauKL}_K^2 \leq C |K| \sum\limits_{n=1}^{d+1} |\bc_n|^2
\leq C |K| \sum\limits_{m=1}^{d+1}
     \sum\limits_{\begin{subarray}{c} n=1\\ n\neq m \end{subarray}}^{d+1}
     |R_{|\gamma_m}(\bfx_n)|^2
\leq C h_K \norm{ R }_{\partial K}^2.
\end{equation}

Similarly, we have
$$
  \norm{\btauKQ}_K^2 \leq C \sum\limits_{n=1}^{d+1}
   \sum\limits_{\begin{subarray}{c} m=2\\ m > n \end{subarray}}^{d+1}
  \int\limits_K \lambda_m^2 \lambda_n^2 \dx |\bt_{mn}|^4 |\bnabla r|^2
\leq C h_K^4 |K|\; |\bnabla r|^2
\leq C h_K^4 \norm{ \bnabla r }_K^2,
$$
where we used the fact that $\bnabla r$ is constant over $K$.
Since $r \in \mathbb{P}^1(K)$, we have the inverse inequality
$\norm{\bnabla r}_K \leq C h_K^{-1} \norm{r}_K$ and we obtain
\begin{equation}
\label{eq:tauKQest1}
  \norm{\btauKQ}_K \leq C h_K \norm{r}_K = C h_K \norm{\Pi_K r_h}_K,
\end{equation}
because $r_h = f - \kappa_K^2 u_h + \Delta u_h$ and
$\Pi_K r_h = r$ on $K$.

Finally, using estimate \ttg{eq:normRest} in \ttg{eq:tauKLest1}
and estimate \ttg{eq:Pirest} in \ttg{eq:tauKQest1}, we derive
\begin{align}
\label{eq:tauKLest2}
  \norm{\btauKL}_K &\leq C \left[
    \trinorm{u - u_h}_{\widetilde K} + h_K \norm{f - \Pi f}_{\widetilde K}
   + h_K^{1/2} \norm{\gN - \Pi^K_\gamma \gN }_{\GammaN\cap\partial K}
 \right],
\\
\label{eq:tauKQest2}
  \norm{\btauKQ}_K &\leq C \left[
    \trinorm{u - u_h}_K + h_K \norm{f - \Pi f}_K
 \right].
\end{align}
Notice that the assumption
$\kappa_K\rho_K \leq 1$ and the shape regularity \ttg{eq:shapereg}
imply the existence of a constant $C>0$ such that
$C h_K \leq \min\{h_K,\kappa_K^{-1}\} \leq h_K$.
Due to Lemma~\ref{le:divtau1} and definition \ttg{eq:etaK} we have
$$
  \eta_K(\btau_K^{(1)}) = \norm{\btau_K^{(1)} - \bnabla u_h}_K
  = \norm{\btauKL + \btauKQ}_K
\leq \norm{\btauKL}_K + \norm{\btauKQ}_K
$$
and estimates \ttg{eq:tauKLest2}--\ttg{eq:tauKQest2} finish the proof.
\end{proof}

\subsection{Flux reconstruction \#2}
On elements $K$ for which $\kappa_K\rho_K > 1$, we use  
a flux reconstruction given by
\begin{equation}
\label{eq:tauK2}
  \btau_K^{(2)} = \bnabla u_h|_K + \btauKO,
\end{equation}
where the vector field $\btauKO$ is defined piecewise on each element $K$ in
the following way.
We consider $d+1$ subsimplices $K_\gamma$ of $K$ that are
defined as convex hulls of the incentre $\bfx_K$ and facets $\gamma$ of $K$.
In each subsimplex $K_\gamma$ we define $\btauKO$ to be
\begin{equation*}
%\label{eq:tauKO}
%\btauKO = \rho_K^{-1} \mathrm{e}^{-\kappa_K x_d} (\bfx - \bfx_K) R(x_1,\dots,x_{d-1})
\btauKO(\bfx) = \rho_K^{-1} (1-\kappa_K x_d)^+ (\bfx - \bfx_K) R(x_1,\dots,x_{d-1})
\quad\text{in }K_\gamma,
\end{equation*}
where $z^+ = (|z|+z)/2$ stands for the positive part of $z$,
$R = g_K - \bnabla u_h|_K \cdot \bn_K$ as before,
$\rho_K$ is the inradius of $K$, and
$\bfx = (x_1,x_2, \dots, x_d)$ are local Cartesian coordinates defined in such
a way that points $(x_1,\dots,x_{d-1})$ lie in the plane of $\gamma$ and
$x_d$ corresponds to the direction perpendicular to $\gamma$ aiming inwards $K$,
see Figure~\ref{fi:Kgamma} for a three-dimensional illustration.

\begin{figure}
\begin{center}
%\newlength{\ww}
\setlength{\ww}{1mm}
\includegraphics[width=60\ww]{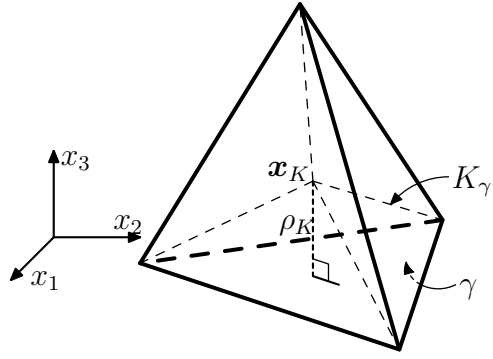}%
\makebox[0pt][l]{\hspace{-57\ww}\raisebox{9\ww}[0pt][0pt]{$x_1$}}%
\makebox[0pt][l]{\hspace{-46\ww}\raisebox{16.5\ww}[0pt][0pt]{$x_2$}}%
\makebox[0pt][l]{\hspace{-53\ww}\raisebox{25\ww}[0pt][0pt]{$x_3$}}%
\makebox[0pt][l]{\hspace{-1\ww}\raisebox{22\ww}[0pt][0pt]{$K_\gamma$}}%
\makebox[0pt][l]{\hspace{0\ww}\raisebox{8\ww}[0pt][0pt]{$\gamma$}}%
\makebox[0pt][l]{\hspace{-25.5\ww}\raisebox{23.5\ww}[0pt][0pt]{$\bfx_K$}}%
\makebox[0pt][l]{\hspace{-24\ww}\raisebox{16.5\ww}[0pt][0pt]{$\rho_K$}}%
\end{center}
\caption{\label{fi:Kgamma}
Division of $K$ into subsimplices $K_\gamma$ and the local Cartesian coordinates.
}
\end{figure}

Clearly, $\btauKO$ vanishes for $x_d \geq \kappa_K^{-1}$.
We also observe that the normal component of $\btauKO$ vanishes on
$\partial K_\gamma \setminus \gamma$ whilst on $\gamma$
$$
  \btauKO \cdot \bn_K |_\gamma
%= \rho_K^{-1} R(x_1,\dots,x_{d-1}) (\bfx - \bfx_K) \cdot \bn_K |_\gamma
%= \rho_K^{-1} R(x_1,\dots,x_{d-1}) \rho_K
= R(x_1,\dots,x_{d-1}),
$$
because $x_d = 0$ on $\gamma$ and $(\bfx - \bfx_K) \cdot \bn_K = \rho_K$
on $\gamma$.
These conditions guarantee that $\btau_K^{(2)} \in \Hdiv[K]$
and $\btau_K^{(2)} \cdot \bn_K = g_K$ on $\partial K$.
Moreover, the flux leads to a locally efficient estimator for the error
in the case $\kappa_K\rho_K > 1$:

\begin{lemma}\label{le:tauK2}
Let $K\in\cT_h$ then
$\btau_K^{(2)} \cdot \bn_K = g_K$ on $\partial K$
and, if $\kappa_K\rho_K > 1$, then
$$
  \eta_K\bigl(\btau_K^{(2)}\bigr) \leq C \left(
    \trinorm{u-u_h}_{\widetilde K}
  + \kappa_K^{-1} \norm{f-\Pi f}_{\widetilde K}
  + \kappa_K^{-1/2} \norm{\gN-\Pi^K_\gamma \gN}_{\GammaN\cap\partial K}
\right).
$$
\end{lemma}
\begin{proof}
The first assertion has been shown above. Suppose $\kappa_K\rho_K > 1$, then
since $|\bfx - \bfx_K| \leq h_K$ and $h_K/\rho_K$ is bounded uniformly
thanks to \ttg{eq:shapereg}, we obtain
\begin{equation}
\label{eq:L2normtauKO}
\norm{\btauKO}_{K_\gamma}^2
\leq \frac{h_K^2}{\rho_K^2}
    %\int_0^\infty \mathrm{e}^{-2\kappa_K x_d} \dx[x_d]
    \int_0^{\kappa_K^{-1}} (1-\kappa_K x_d)^2 \dx[x_d]
    %\int_\gamma R(x_1,\dots,x_{d-1})^2 \dx[x_1]\dots\dx[x_{d-1}]
    \norm{R}_\gamma^2
= \frac{h_K^2}{\rho_K^2} \frac{1}{3\kappa_K} \norm{R}_\gamma^2
\leq C \kappa_K^{-1} \norm{R}_\gamma^2.
\end{equation}
For $x_d \leq \kappa_K^{-1}$, a simple computation yields inequality
$$
  |\ddiv \btauKO| \leq
  %\rho_K^{-1} \mathrm{e}^{-\kappa_K x_d}
  %\left[ (d + \kappa_K h_K) |R| + h_K |\bnabla_\gamma R| \right]
  \rho_K^{-1} (1-\kappa_K x_d)^+
  \left( d |R| + h_K |\bnabla_\gamma R| \right) + \rho_K^{-1} \kappa_K h_K |R|,
  %\quad \text{for } x_d \leq \kappa_K^{-1},
$$
where $\bnabla_\gamma$ denotes the gradient with respect to $x_1,\dots,x_{d-1}$
only.
Consequently,
$$
  \norm{\ddiv \btauKO}_{K_\gamma}^2
% \leq \frac{C}{\rho_K^2} \int_0^\infty \mathrm{e}^{-2\kappa_K x_d} \dx[x_d]
%    \left[ (1+ \kappa_K^2 h_K^2) \norm{R}_\gamma^2
%           + h_K^2 \norm{\bnabla_\gamma R}_\gamma \right].
\leq \frac{C}{\rho_K^2} \left( \int_0^{\kappa_K^{-1}} (1-\kappa_K x_d)^2 \dx[x_d]
   \left[ \norm{R}_\gamma^2 + h_K^2 \norm{\bnabla_\gamma R}_\gamma^2 \right]
    + \kappa_K h_K^2 \norm{R}_\gamma^2
    \right).
$$
Since $R \in \mathbb{P}^1(\gamma)$, we use the shape regularity
and the inverse estimate
$\norm{\bnabla_\gamma R}_\gamma \leq C h_\gamma^{-1} \norm{R}_\gamma$
to derive
\begin{equation}
\label{eq:normdivtauKO}
  \norm{\ddiv \btauKO}_{K_\gamma}^2 \leq
\frac{C}{\rho_K^2} \frac{1}{\kappa_K} \max\{1,\kappa_K h_K \}^2 \norm{R}_\gamma^2
\leq C \kappa_K \norm{R}_\gamma^2,
\end{equation}
where the last inequality follows from
the shape regularity \ttg{eq:shapereg}, from \ttg{eq:shaperegrhoK}, and from
the assumption $\kappa_K\rho_K > 1$.

% this estimate simplifies to
% \begin{equation}
% \label{eq:normdivtauKO}
%   \norm{\ddiv \btauKO}_{K_\gamma}^2 \leq C \kappa_K \norm{R}_\gamma^2.
% \end{equation}

%Thanks to this assumption, we also have $M_{h_K,\kappa_K} = \kappa_K^{-1}$.

Hence, thanks to \ttg{eq:normRest} and \ttg{eq:L2normtauKO}:
\begin{multline*}
  \norm{\btau_K^{(2)} - \bnabla u_h}_K
= \norm{\btauKO}_K
\leq C \kappa_K^{-1/2} \norm{R}_{\partial K}
\\
\leq C \left(
  \trinorm{u - u_h}_{\widetilde K}
+ \kappa_K^{-1} \norm{f - \Pi f}_{\widetilde K}
+ \kappa_K^{-1/2} \norm{\gN - \Pi^K_\gamma \gN}_{\GammaN\cap\partial K}
\right).
\end{multline*}
Similarly, estimates \ttg{eq:normdivtauKO}, \ttg{eq:Pirest}, and \ttg{eq:normRest}
yield the bound
\begin{multline*}
  \kappa_K^{-1} \norm{\Pi_K f - \kappa_K^2 u_h + \ddiv\btau_K^{(2)}}_K
\leq \kappa_K^{-1} \norm{\Pi_K r_h}_K + \kappa_K^{-1} \norm{\ddiv \btauKO}_K
\\
\leq \kappa_K^{-1} \norm{\Pi_K r_h}_K + C \kappa_K^{-1/2} \norm{R}_{\partial K}
\\
\leq C \left(
  \trinorm{u - u_h}_{\widetilde K} + \kappa_K^{-1} \norm{f - \Pi f}_{\widetilde K}
+ \kappa_K^{-1/2} \norm{\gN - \Pi^K_\gamma \gN}_{\GammaN\cap\partial K}
\right).
\end{multline*}
Combining these estimates gives the result claimed.
\end{proof}

\subsection{Main result}

% TO DO: MODIFY ASSUMPTIONS OF THE ABOVE LEMMAS:
% USE $\kappa_K \rhomax \leq 1$ INSTEAD OF $\kappa_K\rho_K \leq 1/\cC_0$.

%We are now in a possition to state our main result.
We combine the flux reconstructions $\btau_K^{(1)}$ and $\btau_K^{(2)}$ in a natural
way and construct $\btau \in \Hdiv$ elementwise as
\begin{equation}
\label{eq:tau}
  \btau|_K = \Bigl\{
  \begin{array}{ll}
   \btau_K^{(1)} & \text{if } \kappa_K \rho_K \leq 1, \\
   \btau_K^{(2)} & \text{if } \kappa_K \rho_K > 1,  %\text{otherwise},
  \end{array}
\end{equation}
where $\btau_K^{(1)}$ and $\btau_K^{(2)}$ are defined in
\ttg{eq:tauK1} and \ttg{eq:tauK2}.
The following theorem shows that the associated 
error estimator provides a guaranteed upper bound
on the error, which is robust with respect to $\kappa$ and $h$.

\begin{theorem}
\label{th:main}
Let $u$ be the exact weak solution given by \ttg{eq:weakf} and
and $u_h \in V_h$ be its finite element approximation \ttg{eq:FEM}.
Let the flux reconstruction $\btau \in \Hdiv$ be given by
\ttg{eq:tau}.
Then the error in $u_h$ is bounded by
$$
% \trinorm{u - u_h}^2 \leq \sum\limits_{K\in\cT_h}
%   \left[ \eta_K(\btau) + \min \left\{ \frac{h_K}{\pi}, \frac{1}{\kappa} \right\}
%    \norm{f - \Pi_K f}_K  \right]^2
\trinorm{u - u_h}^2 \leq \eta^2(\btau)
 = \sum\limits_{K\in\cT_h} \left[ \eta_K(\btau)
  + \osc_K(f) + \osc_{\GammaN\cap\partial K}(\gN) \right]^2.
$$
Moreover, there exists a positive constant $C$, independent of any
mesh-size or any values $\kappa_K$ satisfying \ttg{eq:kappacond1}--\ttg{eq:kappacond2},
such that
\begin{multline*}
  \eta_K(\btau) \leq C \Bigl(
     \trinorm{u-u_h}_{\widetilde K}
   + \min\{h_K,\kappa_K^{-1}\} \norm{f-\Pi f}_{\widetilde K}
\\
   + \min\{h_K,\kappa_K^{-1}\}^{1/2} \norm{\gN-\Pi^K_\gamma \gN}_{\GammaN\cap\partial K}
\Bigr).
\end{multline*}
\end{theorem}
\begin{proof}
It follows immediately from Lemmas~\ref{le:1}, \ref{le:tauK1}, and \ref{le:tauK2}.
\end{proof}

In view of convention \ttg{eq:kappa0}, this result holds even
if $\kappa_K = 0$ for any number of elements $K\in\cT_h$.
Theorem~\ref{th:main} provides a robust, computable upper bound, but it is
possible to improve the bound at the expense of having to compute both
$\eta_K(\btau_K^{(1)})$ and $\eta_K(\btau_K^{(2)})$ on every element.
The associated flux is defined by
\begin{equation}
\label{eq:taustar}
  \btau^*|_K = \Bigl\{
  \begin{array}{ll}
   \btau_K^{(1)} & \text{if } \kappa_K = 0 \text{ or if }
                \eta_K(\btau_K^{(1)}) \leq \eta_K(\btau_K^{(2)}), \\
   \btau_K^{(2)} & \text{otherwise}. % \text{if } \kappa_K \neq 0 \text{ and }
                % \eta_K(\btau_K^{(1)}) > \eta_K(\btau_K^{(2)}), 
  \end{array}
\end{equation}
and the corresponding estimator is given by $\eta(\btau^*)$,
which in turn involves the local indicator
$\eta_K(\btau^*) = \min \bigl\{ \eta_K(\btau_K^{(1)}), \eta_K(\btau_K^{(2)}) \bigr\}$.
This flux reconstruction is slightly more expensive to compute,
but it yields more accurate estimator than $\btau$,
because $\eta_K(\btau^*) \leq \eta_K(\btau)$.
Clearly, if we replace $\btau$ by $\btau^*$ in Theorem~\ref{th:main},
both its statements remain valid.

\section{Numerical example}
\label{se:numex}

This section illustrates numerical performance of the a posteriori error estimators
$\eta(\btau)$ and $\eta(\btau^*)$ for a three dimensional example.
In particular, the example confirms the robustness
of both estimators with respect to the discontinuous reaction coefficient
$\kappa$ and with respect to the mesh size.

%\subsection{Example A}
We consider problem \ttg{eq:modpro} in a cube $\Omega = (-1,1)^3$,
with piecewise constant coefficient $\kappa$
defined by
$$
  \kappa(x_1,x_2,x_3) = \left\{ \begin{array}{ll}
    \kappa_1 & \text{ for } x_1 < 0, \\
    \kappa_2 & \text{ for } x_1 \geq 0, \\
  \end{array} \right.
$$
where $0 < \kappa_1 \leq \kappa_2$ are constants.
The right-hand side is $f = \kappa_1^2$.
Homogeneous Dirichlet boundary conditions are assumed on
$\GammaD = \{ (x_1,x_2,x_3) \in \partial\Omega : x_1 = \pm 1 \}$
and homogeneous Neumann boundary conditions are prescribed on
$\GammaN = \partial\Omega\setminus\GammaD$.

Its exact solution can be expressed as
$$
  u(x_1,x_2,x_3) = \left\{ \begin{array}{ll}
    A_1 \mathrm{e}^{-\kappa_1 x_1} + A_2 \mathrm{e}^{\kappa_1 x_1} + 1
      & \text{ for } x_1 < 0, \\
    A_3 \mathrm{e}^{-\kappa_2 x_1} + A_4 \mathrm{e}^{\kappa_2 x_1}
        + \kappa_1^2/\kappa_2^2
      & \text{ for } x_1 \geq 0, \\
  \end{array} \right.
$$
where constants $A_1, \dots, A_4$ are uniquely determined by the
Dirichlet boundary conditions and by the requirement of $C^1$ continuity
of $u(x_1,x_2,x_3)$ for $x_1 = 0$.
In the subsequent computations we fix $\kappa_2 = 10^6$ and hence the solution
has a boundary layer at least in the vicinity of the face $x_1=1$.
Although the true solution has a univariate nature,
this plays no role in the computations.

We approximate this problem using linear finite elements on uniform tetrahedral meshes
that are constructed in two steps. First, the cube $\Omega$ is uniformly divided
into $M^3$ subcubes and then each subcube is split into 6 tetrahedrons
along its diagonal.
The resulting mesh then has $N_\mathrm{DOF} = (M-1)(M+1)^2$ degrees of freedom.

Figure~\ref{fi:exAerr-kappa} presents the results for a fixed mesh ($M=16$, $N_\mathrm{DOF}=4335$).
%NDOF = 1089
The left panel shows the dependence of the true error $\trinorm{u-u_h}$
and the error estimators $\eta(\btau)$ and $\eta(\btau^*)$
as $\kappa_1$ is varied in the range $(0,\kappa_2]$.
The right panel presents the effectivity indices 
$I_\mathrm{eff} = \eta / \trinorm{u-u_h}$.
We observe that both estimators provide upper bound on the
error and that they robustly capture the behaviour of the error in the whole
range of values of $\kappa_1$. Thus, they are independent of the ratio
$\kappa_1/\kappa_2$ in this case. As expected,
the effectivity index for $\eta(\btau^*)$ is smaller than for $\eta(\btau)$.
Both indices exhibit values around 2 for small values of $\kappa_1$
and they are close to 1 for $\kappa_1 \geq 1000$.

\begin{figure}
%\begin{center}
%\newlength{\ww}
%\setlength{\ww}{1.00mm}
\centerline{
\includegraphics[width=77mm]{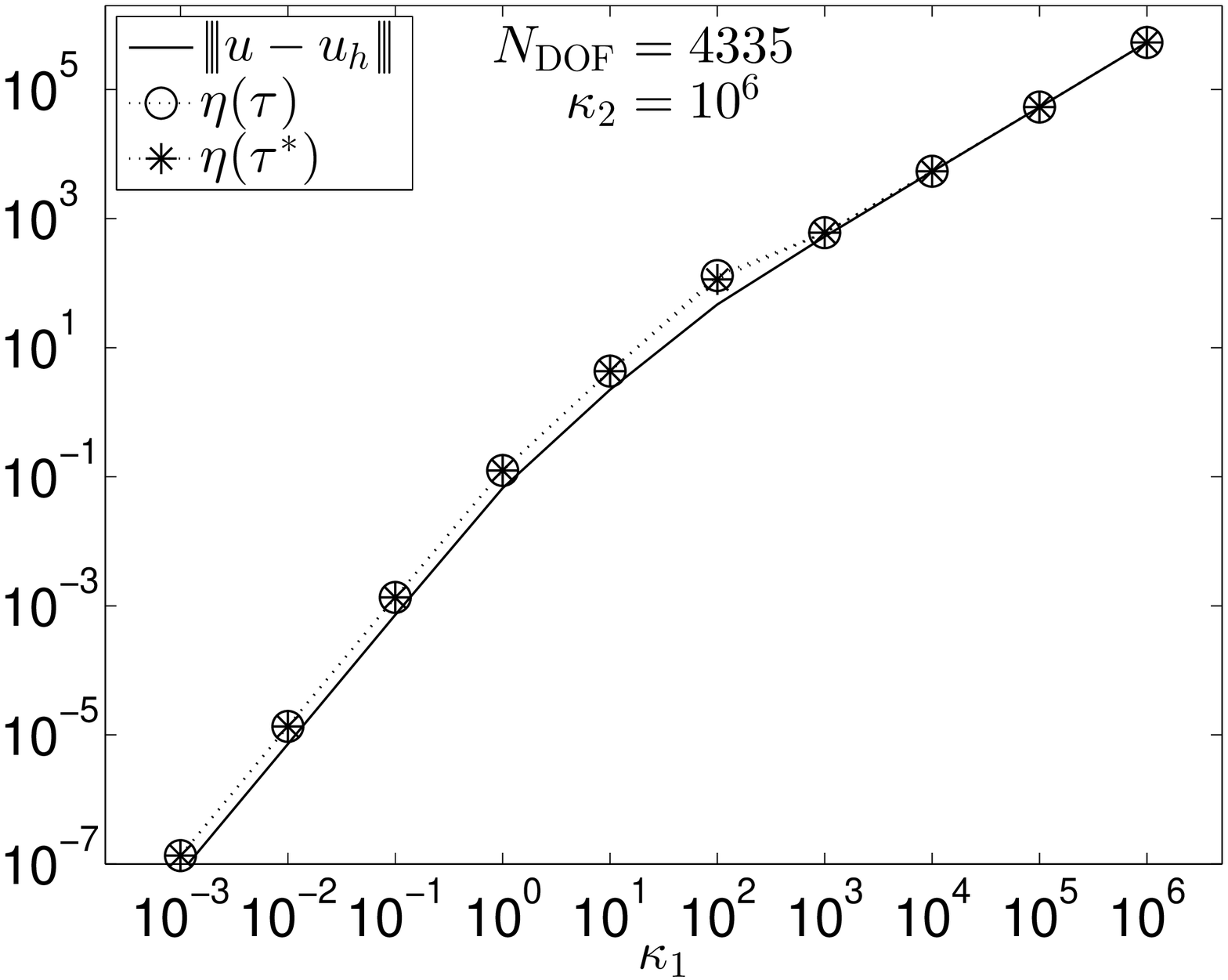}%
\ 
\includegraphics[width=77mm]{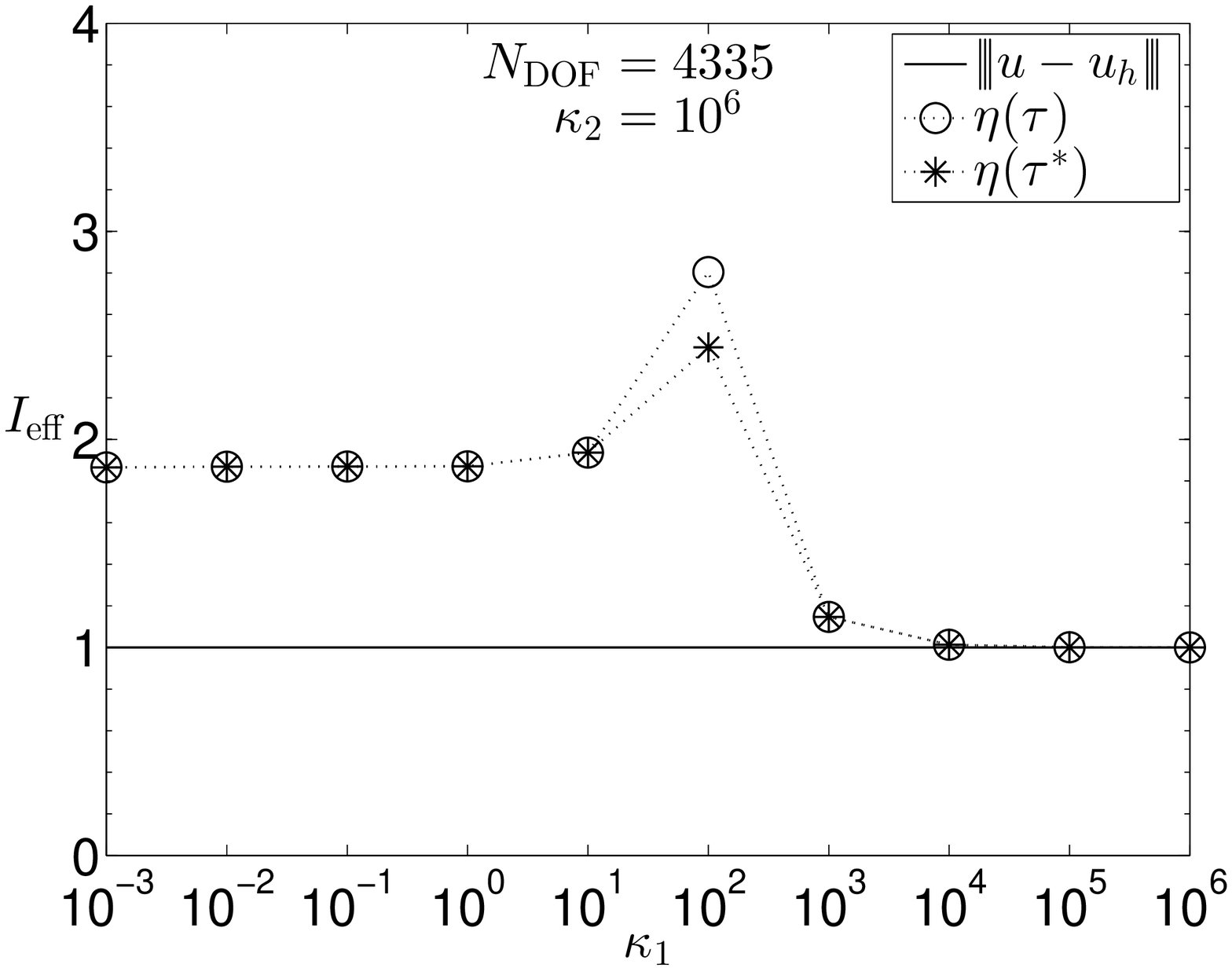}%
}
%\end{center}
\caption{\label{fi:exAerr-kappa}
Dependence of $\trinorm{u-u_h}$, $\eta(\btau)$, and $\eta(\btau^*)$ on $\kappa_1$ (left)
and corresponding effectivity indices (right).
These results correspond to $\kappa_2=10^6$ and to a mesh with
$N_{\mathrm{DOF}} = 4335$ ($M=16$).
}
\end{figure}

Similarly, Figure~\ref{fi:exAerr-ndof} demonstrates the behaviour of
these error estimators and of the true error with respect to the number of degrees
of freedom.
In this case we fix $\kappa_1=100$ and solve the problem on a series
of meshes with $M=2,2^2,2^3, \dots, 2^7$.
We have chosen the most unfavourable value $\kappa_1=100$ for which
both error estimators exhibit the highest overestimation in Figure~\ref{fi:exAerr-kappa}.

As above, the left panel of Figure~\ref{fi:exAerr-ndof}
presents the values of the true error $\trinorm{u-u_h}$
and of the estimators $\eta(\btau)$ and $\eta(\btau^*)$,
while the right panel shows the effectivity indices.
Again, we verify the upper bound property of the estimators and
observe their robust behaviour with respect to the mesh size.
The effectivity indices have values around 1 and 2 with
an exception of the intermediate case, where the mesh size $h$ is comparable
to $1/\kappa_1$.

\begin{figure}
%\begin{center}
%\newlength{\ww}
%\setlength{\ww}{1.00mm}
\centerline{
\includegraphics[width=77mm]{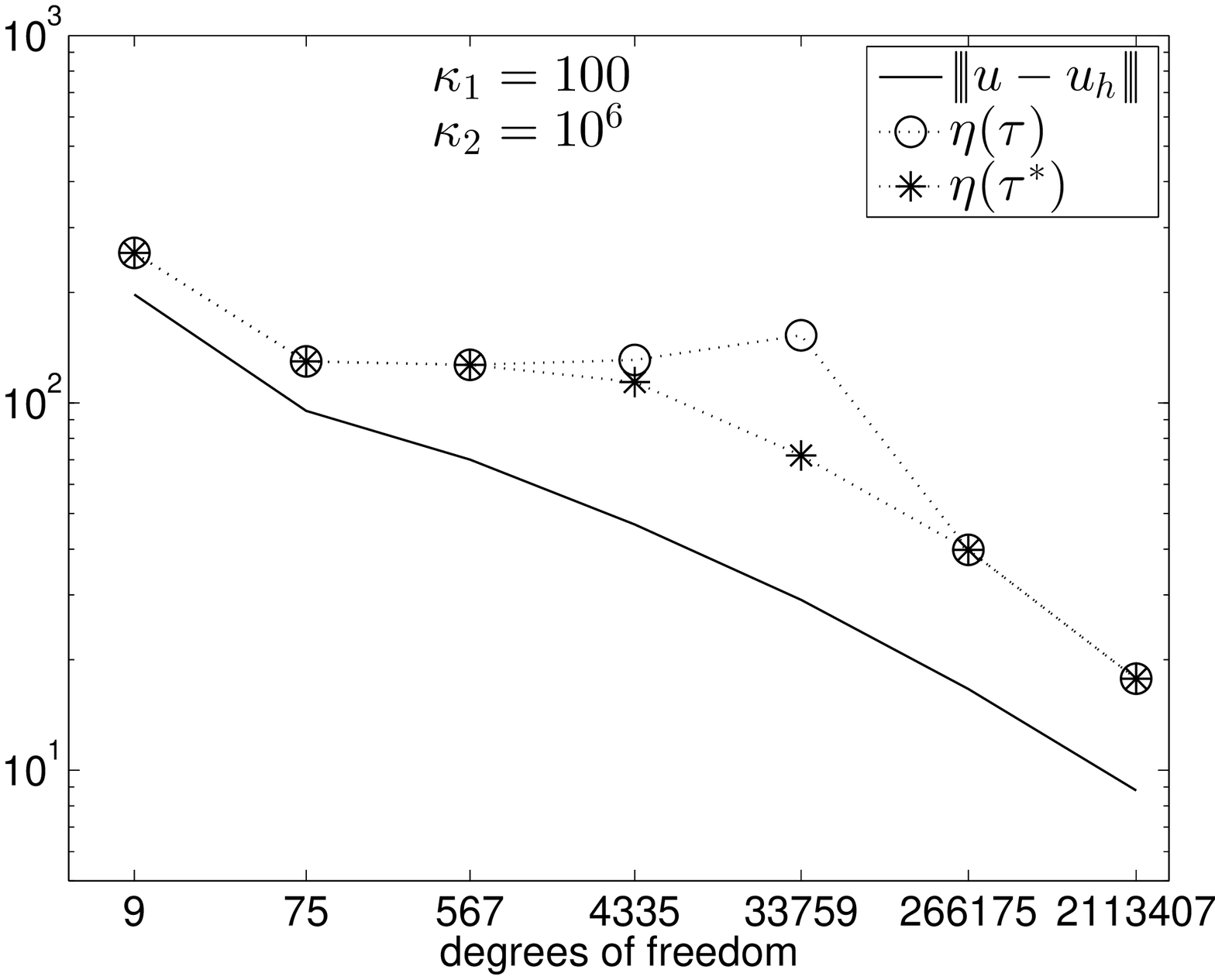}%
%\qquad
\includegraphics[width=77mm]{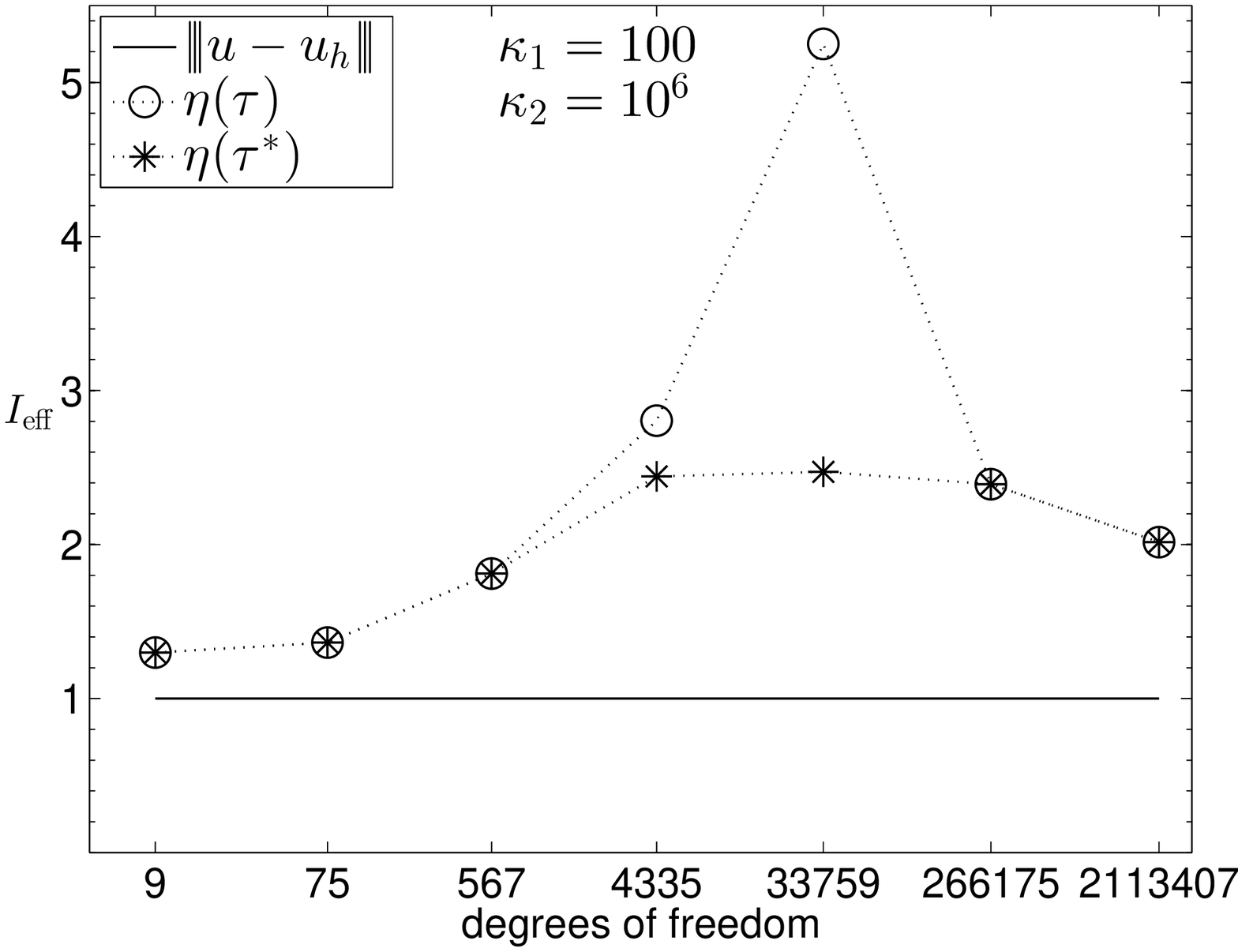}%
}
%\end{center}
\caption{\label{fi:exAerr-ndof}
Dependence of $\trinorm{u-u_h}$, $\eta(\btau)$, and $\eta(\btau^*)$ on
the number of degrees of freedom (left)
and corresponding effectivity indices (right).
These results were computed on a sequence of uniformly refined meshes with $\kappa_1=100$ and $\kappa_2=10^6$.
}
\end{figure}

\section{Conclusions}
\label{se:conclusions}

We presented a robust a posteriori error estimator on the energy norm
of the approximation error for a reaction-diffusion problem in arbitrary dimension.
The reaction coefficient $\kappa$ is assumed to be piecewise constant
and mixed Dirichlet-Neumann boundary conditions are allowed.
The estimator is robust with respect to the reaction coefficient $\kappa$,
including the singularly perturbed case,
and it provides a computable upper bound on the error.
The upper bound is guaranteed up to round-off errors and quadrature errors
in the evaluation of $\eta(\btau)$.

The evaluation of the error estimator can be implemented as a fast algorithm
in the sense that the computational complexity is proportional to the number
of elements. Indeed, the boundary flux equilibration procedure described in
Section~\ref{se:robequilib} is fast, because it is based on solving small systems
on patches. Flux reconstructions \ttg{eq:tauK1} and \ttg{eq:tauK2}
are given by explicit formulas, hence, the only issue is to loop over all
elements and compute the norms in \ttg{eq:etaK}.

The presented approach is suitable for the piecewise linear finite element
approximations. The Galerkin condition \ttg{eq:FEM} is required in order to
guarantee the exact equilibration condition \ttg{eq:equilib1} in case of
small values of the reaction coefficient $\kappa$.
On the other hand, the exact equilibration is not needed for large values
of $\kappa$ and the presented error estimator can be used for an arbitrary
(conforming) approximate solution $u_h \in V$.

% Upper bounds on the error of arbitrary approximations are of particular
% interest, because
% they estimate the total approximation error. This contrasts with many other
% error estimators which bound the discretization error only.
% The most important components of the total approximation error for linear
% elliptic problems are the discretization error
% and the algebraic error (stemming from inexact solution of linear algebraic systems).
% Separate estimates of the discretization and algebraic error
% can be obtain by the recent technique of
% \cite{ErnVoh:2012a,ErnVoh:2012b,JirStrVoh:2010}.
% These separate estimates enable to stop the iterative algebraic solvers
% at the moment when the algebraic error is comparable (or slightly smaller)
% then the discretization error.
% This balancing of algebraic and discretization error may yield considerable
% computational savings.

Finally, we note that whilst we have assumed
conformity of the approximation, this is not essential.
Methodologies derived in \cite{Ainsworth:2005,Ainsworth:2007,ErnSteVoh:2010}
could be used
to extend the error bound to any piecewise linear non-conforming approximation.

% \begin{equation}
% \label{}
% \end{equation}

% \begin{lemma}
% \end{lemma}
% \begin{proof}
% \end{proof}

% \bibliographystyle{panm}
% %\bibliographystyleCa{myunsrt} %{abbrv}
% \bibliography{%
% /home/vejchod/tex/bib/vejchod,%
% /home/vejchod/tex/bib/vejchod_aee,%
% /home/vejchod/tex/bib/citations%
% }



\bigskip
\noindent{\bf References}

% \bibliographystyle{elsarticle-num}
% \bibliography{%
% /home/vejchodsky/tex/bib/vejchod,%
% /home/vejchodsky/tex/bib/vejchod_aee,%
% /home/vejchodsky/tex/bib/citations%
% }


\end{document}